\documentclass{amsart}
\usepackage{amsmath,amsthm,amssymb}
\usepackage{lipsum} 
\usepackage[utf8]{inputenc}
\usepackage{esint}
\usepackage{amsfonts} 
\usepackage{xcolor}
\usepackage{tikz}
\usepackage[colorlinks=true]{hyperref}
\usepackage{bbold}
\usepackage{enumerate}
\usepackage{calc}
\usepackage[style=alphabetic,maxnames=200,maxalphanames=6,giveninits]{biblatex}
\makeatletter
\numberwithin{equation}{section}
\theoremstyle{plain}
        \newtheorem{theorem}{Theorem}[section]
        
        \newtheorem{lemma}[theorem]{Lemma}
        \newtheorem{corollary}[theorem]{Corollary} 
        
        \newtheorem{definition}[theorem]{Definition} 
        \newtheorem{remark}[theorem]{Remark}

        \newtheorem*{claim*}{Claim}

        \newtheorem*{fact*}{Fact} 
        
\newtheorem*{theorem*}{Theorem}
\newtheorem*{definition*}{Definition}
\newtheorem*{proposition*}{Proposition}
\usepackage{graphicx}
\let\oldmarginpar\marginpar
\renewcommand\marginpar[1]{\-\oldmarginpar[\raggedleft\footnotesize #1]
{\raggedright\footnotesize #1}}
\newcommand \be {\begin{equation}}
\newcommand \ee {\end{equation}}
\newcommand \la \langle
\newcommand \ra \rangle

\renewcommand\div{\text{div\,}}
\newcommand \loc {\text{loc}}

\newcommand{\R}{\mathbb{R}}

\newcommand{\N}{\mathbb{N}}
\newcommand{\PP}{\mathbb{P}}
\newcommand{\PPi}{\mathbb{\Pi}}

\newcommand{\cF}{\mathcal{F}}
\newcommand{\cS}{\mathcal{S}}

\renewcommand{\d}{\partial} 
\DeclareMathOperator{\di}{div}
\newcommand{\defeq}{\mathop{=}\limits^{\textrm{def}}}

\title[2D stationary incompressible inhomogeneous Navier--Stokes equations]{Solvability of the two-dimensional stationary incompressible inhomogeneous Navier--Stokes equations with variable viscosity coefficient}

\author[Z. He]{Zihui He} 

\author[X. Liao]{Xian Liao} 
\address[Z. He and X. Liao]
{Institute of Analysis, Karlsruhe Institute of Technology\\
Englerstraße 2, 76131 Karlsruhe, Germany.}
\email{zihui.he@kit.edu, xian.liao@kit.edu}

\date{\today}

\begin{document}

\maketitle

\begin{abstract}
    We show the existence and the regularity properties of (a class of)  weak solutions to the two-dimensional stationary incompressible inhomogeneous  Navier--Stokes equations with density-dependent viscosity coefficients, by analyzing a fourth-order nonlinear elliptic equation for the stream function. 

We formulate the stationary solutions for the parallel, concentric and radial flows respectively, and we give some (ir-)regularity results as well as  some explicit examples in the case of piecewise-constant viscosity coefficients.  
\end{abstract}



\tableofcontents
\section{Main results} 
In this paper we are concerned with  the two-dimensional  stationary inhomogeneous incompressible Navier--Stokes equations
\begin{equation}\label{SNS}
\left\{
\begin{aligned}
&\di(\rho u\otimes u)-\di(\mu Su)+\nabla\Pi= f,\\
&\di u=0,\,\di(\rho u)=0.
\end{aligned}
\right.
\end{equation}
The unknown density function $\rho\geq0$, the unknown velocity vector field $u=(u_1, u_2)^T\in\R^2$ and the unknown pressure $\Pi\in \R$ depend on the spatial variable $x=(x_1, x_2)\in\R^2$. 
The  variable viscosity coefficient depends continuously on the density function
\begin{equation}\label{mu}
\mu=b(\rho)\in  [\mu_\ast, \mu^\ast], \end{equation}
where the lower and upper bounds $\mu_\ast, \mu^\ast$ are two positive constants and    $b\in C(\R; [\mu_\ast, \mu^\ast])$ is a given function. 
The external force $f:\R^2\mapsto\R^2$ is given.

In the above, $\nabla =\begin{pmatrix}
\d_{x_1}\\\d_{x_2}
\end{pmatrix}$,
   $\nabla u=\begin{pmatrix} \d_{x_1}u_1 & \d_{x_2}u_1\\ \d_{x_1}u_2 & \d_{x_2}u_2\end{pmatrix}$ and the symmetric   strain tensor in \eqref{SNS} reads as
$$
Su\defeq (\nabla +\nabla^T) u=\begin{pmatrix} 2\d_{x_1}u_1 & \d_{x_2}u_1+\d_{x_1}u_2\\ \d_{x_2}u_1+\d_{x_1}u_2 & 2\d_{x_2}u_2\end{pmatrix}.
$$
We also denote $\div =\nabla\cdot$,   $u\otimes u=\begin{pmatrix}{u_1}^2&u_1u_2\\u_1u_2&{u_2}^2\end{pmatrix}$, and then
\begin{align*}
&\div(\rho u\otimes u)=\begin{pmatrix}
\partial_{x_1}(\rho {u_1}^2) + \partial_{x_2}(\rho u_1u_2)\\
\partial_{x_1}(\rho u_1u_2)+\partial_{x_2}(\rho {u_2}^2)\end{pmatrix},
\\
&\div(\mu Su)=\begin{pmatrix}
\d_{x_1}(\mu 2\d_{x_1}u_1)+\d_{x_2}(\mu(\d_{x_2}u_1+\d_{x_1}u_2))
\\
\d_{x_1}(\mu(\d_{x_2}u_1+\d_{x_1}u_2))+\d_{x_2}(\mu 2\d_{x_2}u_2)
\end{pmatrix}.
\end{align*}

This section is organised as follows.

In Subsection \ref{subs:works} some  related    incompressible Navier--Stokes models and their mathematical results will be  presented.

In  Subsection \ref{subs:result} we will show the \textit{existence and  regularity} results of (a class of) weak solutions in a bounded domain
  to the stationary Navier--Stokes equations \eqref{SNS}   in Theorem \ref{thm}.
   The proof of Theorem \ref{thm} will be placed in Section \ref{proof}, and for the reason of completeness both the exterior domain case and the whole plane case will be studied in Appendix \ref{app:ex}.

In Subsection \ref{subs:symmetry} we are  interested in the symmetric solutions to the stationary Navier--Stokes equations \eqref{SNS}, and in particular we will formulate the solutions for the \textit{parallel, concentric and radial flows} respectively in Theorem \ref{thm2}.
In the case of piecewise-constant viscosity coefficients, we  will  give some \textit{irregularity} results in Corollary \ref{corollary}.
Some explicit examples will be given in Appendix \ref{example}.

In Subsection \ref{subs:discussion}, we will  show the \textit{$L^p$-type regularity} of   the velocity gradient under some specific  regularity assumptions on the viscosity coefficient (where the piecewise-constant viscosity coefficient case is allowed) in Theorem \ref{thm3}.

We remark here that in the proofs of existence and regularity results, we are going to consider the stream function $\Phi$ such that $u=\nabla^\perp\Phi=\begin{pmatrix}\d_{x_2}\Phi\\ -\d_{x_1}\Phi\end{pmatrix}$, which satisfies a fourth-order elliptic equation with   the fourth-order elliptic operator given by
$$
L_\mu=(\d_{x_2x_2}-\d_{x_1x_1}) \mu (\d_{x_2x_2}-\d_{x_1x_1}) +(2\d_{x_1x_2})\mu (2\d_{x_1x_2} ) .
$$ 
On the other side, the irregularity of the viscosity coefficient may result in irregularity on the velocity vector field, e.g. the piecewise-constant viscosity coefficient may imply $\Delta u\not\in L^1_\loc(\R^2)$ or $\div(\mu Su)\not\in L^1_\loc(\R^2)$ (see Corollary \ref{corollary}), although $\nabla u\in L^p(\R^2)$ and
the divergence-free part of the viscous term $\PP\div(\mu Su)=\nabla^\perp \Delta^{-1} L_\mu \Phi\in L^p(\R^2)$   for all finite $p$ (see Theorem \ref{thm3}).

\subsection{Related works}\label{subs:works}
There are a few  works in the literature contributing to the study of the {\it evolutionary} two-dimensional incompressible inhomogeneous Navier--Stokes equations with variable viscosity coefficient
\begin{equation}\label{NS}
\left\{
\begin{aligned}
&\partial_{t}(\rho u)+\operatorname{div}(\rho u\otimes u)-\operatorname{div}(2 \mu Su)+\nabla \Pi=0,\quad(t, x) \in \mathbb{R}^{+} \times \Omega,\\
&\operatorname{div} u=0,\quad \partial_{t} \rho+\operatorname{div}(\rho u)=0,\\
&\rho\big|_{t=0}=\rho_{0},\quad (\rho u)\big|_{t=0}=m_{0}.
\end{aligned}
\right.
\end{equation}
\textcite{Lions} showed the global-in-time existence of   weak solutions
$(\rho, u)\in \bigl(L^\infty(\R^+\times\Omega),  (L^2(\R^+; H^1(\Omega)))^2\bigr)$   of the system \eqref{NS} under the initial condition $\rho_0\in L^\infty(\Omega)$, $\frac{m_0}{\rho_0}\in (L^2(\Omega))^2$. 
The uniqueness and  the regularity properties  of such weak solutions are still open, even in dimension two. There are some partial results toward this issue, but to our best knowledge they are all limited to the case where  the viscosity coefficient $\mu(x)$ is close to some positive constant $\nu\in\R^+$:
\begin{equation}\label{munu}
\|\mu(x)-\nu\|_{L^\infty(\Omega)}<\varepsilon,
\end{equation}
where $\varepsilon$ is some small enough positive constant. \textcite{Desjardins} showed the regularity property of the velocity vector field $u\in L^\infty(\R^+; (H^1(\mathbb{T}^2))^2)$ for initial data $u|_{t=0}\in (H^1(\mathbb{T^2}))^2$, 
 if the smallness condition \eqref{munu} holds.
 \textcite{Zhang1} proved the existence and uniqueness of the solution under \eqref{munu} and further smoothness assumptions on the initial density function $\rho_{0}-1 \in L^{2}(\R^2)\cap L^{\infty} \cap \dot{W}^{1, r}(\R^2)$, $r>2$. 
  \textcite{Paicu} considered the  so-called density-patch problem with piecewise-constant density function $\rho_0=\eta_1\mathbb 1_{\Omega}(x)+\eta_2\mathbb 1_{\Omega^C}(x)$, $\eta_1, \eta_2\in\R^+$, and showed that the $H^3(\R^2)$-boundary regularity of the domain is propagated by   time evolution provided with \eqref{munu}. 
 The case where $\mu(x)=\nu$ is a positive constant has been intensively studied in the past two decades, see e.g. \cite{Danchin,DanchinMucha, Ladyzhenskaya1} and the references therein. 
 It is also worth mentioning the work \cite{Kazhikov} for the study of the {\emph {compressible}} Navier--Stokes equations with  variable viscosity coefficient.

If we consider the stationary \emph{homogeneous} incompressible flow where the density function $\rho=1$ and the viscosity coefficient $\mu=\nu$ is a positive constant, then the  system \eqref{SNS} becomes the following classical stationary Navier--Stokes equations  
\begin{equation}\label{CNS}
\left\{
\begin{aligned}
&\di(  u\otimes u)-\nu \Delta u+\nabla\Pi= f,\quad x\in \Omega,\\
&\di u=0.
\end{aligned}
\right.
\end{equation}
It has been studied extensively in the literature, whenever the underlying domain is a connected bounded domain $\Omega$ , or a multi-connected domain $\cup_{i=1}^n\Omega_i$, or the exterior of a multi-connected set $U=(\cup_{i=1}^n\Omega_i)^C$, or   the whole plane $\R^2$, see the celebrated books \cite{Galdi,Ladyzhenskaya}.
If $\Omega$ has a boundary $\d\Omega$, we assume the boundary value condition for the system \eqref{CNS}
\begin{equation}\label{u00}
u|_{\d \Omega}=u_0,
\end{equation}
and in compatible with $\div u=0$ we assume   no   flux through the boundary $\d\Omega$ 
\begin{equation}\label{u0}
 \int_{\partial\Omega}u_0\cdot n\,ds=0.
\end{equation}
In the above, $n=(n_1, n_2)$ denotes the exterior normal vector to the boundary $\partial\Omega$. 

 \textcite{Leray} showed the existence of weak solutions $u\in (H^1(\Omega))^2$  on a  simply connected  bounded domain $\Omega$ under the zero flux condition \eqref{u0}. 
 This solvability result can be generalized straightforward to a multi-connected domain case $\cup_{i=1}^n\Omega_i$, if we assume  furthermore no flux through the boundary of each connected component
\begin{equation}\label{ui}
\mathcal F_i=\int_{\partial\Omega_i}u_0\cdot n\,ds=0,\quad \forall 1\le i\le n.
\end{equation}
If we assume only the smallness of the fluxes $\mathcal F_i$  
or assume some further symmetric properties, the solvability of the system \eqref{CNS} was also obtained, cf. \cite{Galdi4}.
 On a multi-connected domain  with only the zero \emph{total flux} condition \eqref{u0}, the solvability was shown by \textcite{Korobkov}. 
\textcite{Leray} studied  the system \eqref{CNS} also on the exterior domain   of  a multi-connected set $U=(\cup_{i=1}^n\Omega_i)^C$ under the boundary condition \eqref{ui}, and obtained the   weak solutions $u\in (\dot{H}^1(U))^2$  by constructing a sequence of weak solutions on the bounded domains which converge to $U$.
If the fluxes $\cF_i$ are small, 
the solvability of \eqref{CNS} on $U$ was established by \textcite{Finn}. 
Concerning the whole plane $\R^2$ case, \textcite{Guillod} showed that for any given vector $d\in\R^2$ and a bounded positive measure set $D\subset \R^2$, there exist solutions $u\in (\dot{H}^1(\R^2))^2$ satisfying the prescribed mean value on $D$:  $d=\frac{1}{\hbox{meas}(D)}\int_D u\in \R^2$. 
However, the existence of  decaying solutions, as well as the uniqueness and the asymptotic behaviour of the solutions on the unbounded domains are still open, see e.g.  \cite{Guillod2,Guillod4,Russo} for further related discussions.
We also mention that \textcite{Leray} studied also \eqref{CNS}  
 in dimension three, as well as   the \emph{evolutionary} classical Navier-Stokes equations  (i.e. \eqref{NS} with $\rho=1$ and $\mu=\nu$, see also  the celebrated books \cite{ConstantinFoias, Temam}).

The stationary  inhomogeneous incompressible flow with \emph{constant} viscosity coefficient  is described by
\begin{equation}\label{NNS}
\left\{
\begin{aligned}
&\di(\rho u\otimes u)-\nu \Delta u+\nabla\Pi= f,\quad x\in\Omega,\\
&\di u=0,\,\di(\rho u)=0.
\end{aligned}
\right.
\end{equation}
On a simply connected domain in dimension two,  by using the incompressibility condition $\di u=0$ and the zero flux condition \eqref{u0}, the velocity vector field $u$ can be written as 
$$u=\nabla^\perp \Phi,$$ 
where $\Phi$ is the stream function of $u$. If 
$$\rho=\eta(\Phi),$$
for some well-chosen bounded function $\eta\in L^\infty(\R; [0,\infty))$, then the density equation $\di(\rho u)=0$  is automatically satisfied, at least in the distribution sense if $u\in H^1(\Omega)$ \footnote{Let $\{\eta^\varepsilon\}$ be the regularised function sequence of $\eta$ which are uniformly bounded. Then 
$$
\div(\eta^\varepsilon(\Phi) u)=(\eta^\varepsilon)'(\Phi)\nabla\Phi\cdot u=0.
$$
The equality $\div(\rho u)=0$ follows as $\{\eta^\varepsilon(\Phi)u\}_\varepsilon$ (up to a subsequence) converges weakly to $\eta(\Phi)u$    in $L^2(\Omega)$.}. \textcite{Frolov} showed the existence and regularity results for the solutions of the following form
\begin{equation}\label{solution0}
(\rho, u)=(\eta(\Phi), \nabla^\perp\Phi),
\end{equation}
where $\eta$ is a \emph{H\"older continuous} function.
From now on, we call  the form \eqref{solution0} as  Frolov's form. 
\textcite{Santos1} improved this existence result to only \emph{bounded}  $\eta$-functions.
Ammar-Khodja and Santos \cite{Santos3,Santos2} considered the unbounded Y-shape domain.

However, to the best of our knowledge, there are neither  existence nor regularity results of solutions to  the two-dimensional stationary Navier-Stokes system \eqref{SNS} with \emph{variable} viscosity coefficient.
We are going to give   some existence and regularity results for the solutions of Frolov's form to the system \eqref{SNS} in Subsection \ref{subs:result}, whose proof is postponed in Section \ref{proof}.
We will  formulate  the solutions with certain symmetry properties in Subsection \ref{subs:symmetry}, where some irregularity results  with piecewise-constant viscosity coefficients will  also be given.
Finally we will discuss further  regularity issues in Subsection \ref{subs:discussion}.
 
\subsection{Existence and regularity results}\label{subs:result}
We are going to study  the boundary value problem for the  two dimensional stationary inhomogeneous incompressible Navier-Stokes equation \eqref{SNS} on a simply connected bounded $C^{1,1}$-domain $\Omega\subset\R^2$:
\begin{equation}\label{BVP}
\left\{
\begin{aligned}
&\di(\rho u\otimes u)-\di(\mu Su)+\nabla\Pi= f,\\
&\di u=0,\,\di(\rho u)=0,\\
&u|_{\d\Omega}=u_0,
\end{aligned}
\right.
\end{equation}
where the boundary value $u_0$ satisfies the zero flux condition \eqref{u0}: $
 \int_{\partial\Omega}u_0\cdot n\,ds=0$.
Here the positive  viscosity coefficient  $\mu$ depends smoothly on the density function: 
$$\mu=b(\rho),$$
 where $b\in C(\R; [\mu_\ast, \mu^\ast])$ with two positive constants $\mu_\ast, \mu^\ast>0$ is a given function. 
 
 This section is organised as follows.
 We will introduce the boundary value problem for the stream function in Subsection \ref{subs:Phi}.
 Then the density function as well as its boundary condition will be discussed in  Subsection \ref{subs:rho}.
 The definitions of weak solutions and its existence and regularity properties will be given in Subsection \ref{subs:weak} and Subsection \ref{subs:ex,re} respectively.
 The proofs will be postponed in Section \ref{proof}.
 The exterior domain and the whole plane cases will be studied in Appendix \ref{app:ex}.
\subsubsection{Stream function}\label{subs:Phi}
We are going to look for the unknown stream function $\Phi:\overline{\Omega}\mapsto\R$ such that the divergence-free velocity vector field is given by
\begin{equation}\label{Phi}
u=\nabla^\perp \Phi\defeq\begin{pmatrix}\d_{x_2}\Phi\\-\d_{x_1}\Phi\end{pmatrix}.
\end{equation}
Let  $n=(n_1, n_2)$ and $\tau=(n_2, -n_1)$ denote  the unit normal and tangential vector field on the boundary $\partial\Omega$ respectively.
Then $\Phi$ should satisfy the boundary value condition
\begin{equation*}
\frac{\d \Phi}{\d n}\big|_{\d\Omega}=u_0\cdot\tau,
\quad\frac{\d \Phi}{\d\tau}\big|_{\d\Omega}=-u_0\cdot n.
\end{equation*}
If we parameterize the boundary ${\partial\Omega}$ by $\gamma: [0,2\pi)\mapsto{\partial\Omega}$ such that $\gamma'(s)=\tau(\gamma(s))$, then with a constant $C_0\in\R$,
 \begin{equation}\label{Phi0}
 \begin{split}
&\Phi|_{{\partial\Omega}}(\gamma(s))=\Phi_0(\gamma(s))\defeq-\int_{0}^s u_0\cdot n\, d\theta+C_0,\quad s\in [0,2\pi),
\\
&\frac{\d \Phi}{\d n}\big|_{\d\Omega}(\gamma(s))=\Phi_1(\gamma(s))\defeq (u_0\cdot\tau)(\gamma(s)),\quad s\in [0,2\pi).
\end{split}
\end{equation}
We fix this constant $C_0=0$ from now on.

We next derive the partial differential equation satisfied by the stream function.    
We apply $\nabla^\perp\cdot=\begin{pmatrix}\d_{x_2}\\  -\d_{x_1}\end{pmatrix}\cdot\,$ to the first equation in \eqref{SNS} to arrive at
\begin{equation*}
\nabla^\perp\cdot\div(\mu Su)
=-\nabla^\perp\cdot   f+\nabla^\perp\cdot\div(\rho u\otimes u).
\end{equation*}
By straightforward calculation,  the left-hand side reads as a fourth-order elliptic operator with positive variable coefficient $\mu\geq\mu_\ast>0$ on $\Phi$:
\begin{align}\label{PP,Lmu}
\nabla^\perp\cdot\div(\mu Su)
&=\nabla^\perp\cdot\div\left(\mu
\begin{pmatrix}2\d_{x_1x_2}\Phi&(\d_{x_2x_2}-\d_{x_1x_1})\Phi\\
(\d_{x_2x_2}-\d_{x_1x_1})\Phi&-2\d_{x_1x_2}\Phi\end{pmatrix}
\right)\notag\\
&=(\d_{x_2x_2}-\d_{x_1x_1})\bigl( \mu (\d_{x_2x_2}-\d_{x_1x_1})\Phi\bigr)+2\d_{x_1x_2}(\mu 2\d_{x_1x_2}\Phi).
\end{align}   
That is, the first equation in \eqref{SNS} becomes 
\begin{align}\label{elliptic}
&L_\mu\Phi=-\nabla^\perp\cdot f+\nabla^\perp\cdot\div(\rho \nabla^\perp\Phi\otimes\nabla^\perp\Phi),
\end{align}  
where  $L_\mu$ denotes the fourth-order elliptic operator  
\begin{equation}\label{Lmu-}
L_\mu=(\d_{x_2x_2}-\d_{x_1x_1}) \mu (\d_{x_2x_2}-\d_{x_1x_1}) +(2\d_{x_1x_2})\mu (2\d_{x_1x_2} ).
\end{equation}
In particular, if $\mu(x)=\nu$ is a positive constant, then $L_\nu=\nu\Delta^2$. 
We are going to consider the boundary value problem for $\Phi$: 
\begin{equation}\label{BVP:Phi}
\left\{
\begin{aligned}
&L_\mu\Phi=-\nabla^\perp\cdot f+\nabla^\perp\cdot\div(\rho \nabla^\perp\Phi\otimes\nabla^\perp\Phi),\\
& \Phi|_{{\partial\Omega}} =\Phi_0,\quad 
 \frac{\d \Phi}{\d n}\big|_{\d\Omega} =\Phi_1,
\end{aligned}
\right.
\end{equation}
where $L_\mu, \Phi_0, \Phi_1$ are given in \eqref{Lmu-} and \eqref{Phi0} respectively.

We recall here the definition of elliptic operators in divergence form of order $2m$, $m\in\mathbb{N}$ (see e.g. \cite{ADN59,ADN2,Dong}) for readers' convenience.  
 Let $\mathcal L u=\sum_{|\alpha|,|\beta|\le m} D^\alpha (a_{\alpha\beta}D^\beta u)$ where $\alpha$ and $\beta$ are multi-indices,  $u:\R^d\to\R^n$ is a vector-valued function and $a_{\alpha \beta}=\left[a_{\alpha \beta}^{i j}(x)\right]_{i, j=1}^{n}$,  $|\alpha|,|\beta| \le m$, are $n \times n$
matrix-valued functions. We say that $\mathcal L$ is an elliptic operator of $2m$th-order if there exists a constant $\delta\in (0,1)$ such that
$$
\delta|\xi|^{2} \leq \sum_{|\alpha|=|\beta|=m} 
\hbox{Re}\Bigl( a_{\alpha \beta}(x) \xi_{\beta},  \xi_{\alpha}\Bigr)\leq \delta^{-1}|\xi|^2,
$$
for any $x \in \mathbb{R}^{d}$ and $\xi=\left(\xi_{\alpha}\right)_{|\alpha|=m}, \xi_{\alpha} \in \mathbb{R}^{n}$.
Here we can rewrite $L_\mu$ as 
\begin{align*}
L_\mu=&\d_{x_1x_1}\mu \d_{x_1x_1}+\d_{x_2x_2}\mu \d_{x_2x_2}-\d_{x_1x_1}(\mu-\frac{\mu_*}{2}) \d_{x_2x_2}-\d_{x_2x_2}(\mu-\frac{\mu_*}{2}) \d_{x_1x_1}\\
&+2\d_{x_1x_2}(\mu-\frac{\mu_*}{2}) \d_{x_1x_2}+2\d_{x_2x_1}\mu \d_{x_2x_1}
\defeq \sum_{|\alpha|=|\beta|=2}D^\alpha(a^\mu_{\alpha\beta}D^\beta),
\end{align*}
where $\mu_\ast, \mu^\ast>0$ are the positive  lower and upper bounds for the function $\mu$.
Then for any $\xi=\left(\xi_{\alpha}\right)_{|\alpha|=2},\, \xi_{\alpha} \in \mathbb{R}^{2}$ the following inequality holds:
\begin{align*}
\frac{\mu_*}{2}|\xi|^{2} &\leq \sum_{|\alpha|=|\beta|=2} a_{\alpha \beta}^\mu(x) \xi_{\beta} \xi_{\alpha}\\
& =\frac{\mu_*}{2}(\xi_{11}^2+\xi_{22}^2)+(\mu-\frac{\mu_*}{2})(\xi_{11}-\xi_{22})^2+2(\mu-\frac{\mu_*}{2})\xi_{12}^2+2\mu \xi_{21}^2
\le 2 \mu^*|\xi|^ 2.
\end{align*}
Hence, $L_\mu$ is a fourth-order elliptic operator as we can simply take $\delta=\min\{\frac{\mu_\ast}{2}, \frac{1}{2\mu^\ast}, \frac12\}$.

\subsubsection{Density function}\label{subs:rho} 
Following Frolov's idea in \cite{Frolov}, we  can make an Ansatz
\begin{equation*}\label{rho}
\rho=\eta(\Phi),
\end{equation*}
where $\eta\in L^\infty(\R;[\rho_\ast,\rho^\ast])$ with $0\leq\rho_\ast\leq\rho^\ast$ is a nonnegative bounded fixed function, such that  the equation 
$$\div(\rho u)=\div(\eta(\Phi)\nabla^\perp\Phi)=0$$
holds in the distribution sense provided with e.g. $\Phi\in H^2_{\textrm{loc}}(\Omega)$. 

The choice of the function $\eta$ may depend on the prescribed boundary conditions  
$$
\rho|_{\Sigma}=\rho_0,\quad u|_{\d\Omega}=u_0,
$$
where $\Sigma\subset \d\Omega$ is a subset of the boundary $\d\Omega$. 
In the two-river-junction model in \cite{Santos1}, Santos assumed the boundary condition  of the density function as follows
\begin{equation}\label{BC:rho}
\rho|_{\sigma_j}=\rho_j\in L^\infty(\sigma_j),\quad j=1,2,
\end{equation}
where 
$$\sigma_1=\gamma((0, s_1))\hbox{ and }\sigma_2=\gamma((s_2, s_3))\subset\d\Omega=\gamma([0,2\pi)),
\quad 0<s_1<s_2<s_3<2\pi$$
 are two disadjoint boundary sets, such that
\begin{equation*}
(u_0\cdot n)|_{\sigma_1\cup \sigma_2}<0 \hbox{ and }(u_0\cdot n)|_{\gamma([s_1, s_2])}=0.
\end{equation*}
Then the boundary value $\Phi_0$ given in \eqref{Phi0}: $\Phi_0(\gamma(s))=-\int_0^s u_0\cdot nd\theta$ is strictly increasing on $\sigma_1$ and $\sigma_2$, while is constant on the interval $\gamma([s_1, s_2])$.
We can then choose $\eta\in L^\infty(\R; [0,\infty))$ such that 
$$
\eta=\rho_j\circ \Phi_0^{-1} (\hbox{ i.e. }\eta\circ\Phi_0=\rho_j) \hbox{ on }\sigma_j,\quad j=1,2.
$$ 

If $\Sigma=\emptyset$, that is, we do not assume any boundary condition on the density function, we can take any fixed bounded function $\eta\in L^\infty(\R; [0,\infty))$.
We are going to consider this case in this article, and we will search for the solutions for \eqref{BVP} of Frolov's form \eqref{solution0} below
\begin{equation}\label{Fro}(\rho, u)=\bigl(\eta(\Phi), \,\nabla^\perp \Phi\bigr).
\end{equation}
Our results can be generalised straightforward to the above two-river-junction models, which we omit here for the simplicity of presentation. 

\subsubsection{Weak solutions}\label{subs:weak}
In this paragraph we give the definitions of weak solutions to the boundary value problems \eqref{BVP} and \eqref{BVP:Phi} respectively.
We first recall the trace theorem and inverse trace theorem (see e.g. \cite[Section 4.2]{Hsiao}) below.
\begin{theorem}[Trace theorem \& Inverse trace theorem]\label{trace} 
 \begin{enumerate}
 \item 
 Let $\Omega$ be a $C^1$-domain. Then there exists a linear continuous trace operator 
\begin{align*}
\gamma_0: H^1(\Omega)\mapsto H^{\frac12}(\d\Omega), 
\end{align*}
which is an extension of  $\gamma_0 u=u|_{\d\Omega} \hbox{ for }u\in C^0(\overline\Omega)$,
and   there exists a linear continuous right inverse $\Gamma_0$ to $\gamma_0$ with
\begin{align*}
&\Gamma_0:  H^{\frac12}(\d\Omega)\mapsto H^1(\Omega)
\hbox{ and }\gamma_0(\Gamma_0 (u_0))=u_0,
 \hbox{ for all }u_0\in   H^{\frac12}(\d\Omega).
\end{align*}

 \item
Let  $\Omega$ be a $C^{1,1}$ -domain. 
Then there exist two linear continuous trace operators 
\begin{align*}
\gamma_0: H^2(\Omega)\mapsto H^{\frac32}(\d\Omega),\quad \gamma_1: H^2(\Omega)\mapsto H^{\frac12}(\d\Omega),
\end{align*}
which are extensions of 
\begin{align*}
\gamma_0 \Phi=\Phi|_{\d\Omega} \hbox{ for }\Phi\in C^0(\overline\Omega),
\quad \gamma_1 \Phi=\frac{\d\Phi}{\d n}\Big|_{\d\Omega} \hbox{ for }\Phi \in C^3(\overline\Omega).
\end{align*} 
Inversely, there exists a linear continuous right inverse $\Gamma_1$ to $(\gamma_0, \gamma_1)$ with
\begin{align*}
&\Gamma_1: H^{\frac32}(\d\Omega)\times H^{\frac12}(\d\Omega)\mapsto H^2(\Omega)
\hbox{ and }\gamma_j(\Gamma _1(\Phi_0, \Phi_1))=\Phi_j,\quad j=0,1,
\\
&\hbox{ for all }(\Phi_0, \Phi_1)\in H^{\frac32}(\d\Omega)\times H^{\frac12}(\d\Omega).
\end{align*}
 \end{enumerate}
\end{theorem}

 
\begin{definition}[Weak solutions of the Navier--Stokes equations]\label{weaksolution}
Let $\Omega\subset\R^2$ be a bounded simply connected $C^{1}$ domain. We say that a pair $(\rho, u)\in L^\infty(\Omega;[0,\infty))\times H^1(\Omega;\R^2)$ is a weak solution of the boundary value problem \eqref{BVP}  with the given data $u_0\in H^{\frac12}({\partial\Omega};\R^2)$, $f\in H^{-1}(\Omega;\R^2)$, if $\div u=0$, $\div(\rho u)=0$ hold in $\Omega$ in the distribution sense, $u_0=u|_{\partial\Omega}$ is the trace of $u$ on $\partial\Omega$ and the following integral identity
\begin{equation}\label{SNSweak}
\frac12\int_{\Omega}\mu Su:Sv\,dx=\int_{\Omega} \rho( u\otimes u):\nabla v\,dx+\langle f, v\rangle_{H^{-1}(\Omega), H^1_0(\Omega)},
\end{equation} 
holds for all $v\in H^1_0(\Omega; \R^2)$ with $\div v=0$.  
Here  $A:B\defeq\sum_{i,j=1}^2 A_{ij}B_{ij}$ for the matrices $A=(A_{ij})_{1\leq i,j\leq 2}$ and $B=(B_{ij})_{1\leq i,j\leq 2}$.
\end{definition}
 
\begin{definition}[Weak solutions of the   elliptic equation]
Let $\Omega\subset\R^2$ be a bounded simply connected $C^{1,1}$ domain.
Let $\eta\in L^\infty(\R;[0,\infty))$ and $b\in C(\R;[\mu_\ast, \mu^\ast])$ be two given functions.

We say that   $\Phi\in H^2(\Omega)$ is a weak solution of the boundary value problem \eqref{BVP:Phi} with the given data $\Phi_0\in H^{\frac32}(\partial\Omega)$, $\Phi_1\in H^{\frac12}(\d\Omega)$, $f\in H^{-1}(\Omega;\R^2)$, if  
$\Phi_0=\Phi|_{\partial\Omega}$ and $\Phi_1=\frac{\d\Phi}{\d n}\Big|_{\d\Omega}$ in the trace sense and the following integral identity
\begin{equation}\label{ellipticweak}\begin{split}
&\int_{\Omega}\mu
\Bigl( (\d_{x_2x_2}\Phi-\d_{x_1x_1}\Phi) (\d_{x_2x_2}\psi-\d_{x_1x_1}\psi)
+ (2\d_{x_1x_2}\Phi)(2\d_{x_1x_2}\psi)\Bigr)\,dx
\\
&= \langle f,  \nabla^\perp\psi\rangle_{H^{-1}(\Omega), H^1_0(\Omega)}+\int_{\Omega}\rho(\nabla^\perp\Phi\otimes\nabla^\perp\Phi):\nabla\nabla^\perp\psi\,dx,
\end{split}
\end{equation} 
holds for all $\psi\in H^2_0(\Omega)$, where $\rho=\eta(\Phi)$ and $\mu=b(\rho)$.
\end{definition}

 Since for any $v\in H_0^1(\Omega; \R^2)$ defined on a bounded connected $C^{1,1}$ domain $\Omega$   with $\div v=0$, there exists a corresponding stream function $\psi\in H_0^2(\Omega)$   such that $v=\nabla^\perp\psi$, the equality \eqref{ellipticweak} hence implies the equality \eqref{SNSweak} with $u=\nabla^\perp\Phi$.  
 Therefore we have the following fact which passes the solvability of the elliptic equation \eqref{BVP:Phi} to the solvability of the Navier-Stokes system \eqref{BVP}.  
\begin{lemma}\label{fact}
Let $\Omega\subset\R^2$ be a bounded connected $C^{1,1}$ domain. 
Let  $\eta\in L^\infty(\R;[0,\infty))$, $b\in C(\R;[\mu_\ast, \mu^\ast])$ with $0<\mu_\ast\leq \mu^\ast$ and $f\in H^{-1}(\Omega; \R^2)$ be  given.
Let $u_0\in H^{\frac12}(\d\Omega;\R^2)$ satisfy zero-flux condition  \eqref{u0} and let $\Phi_0\in H^{\frac32}(\d\Omega), \Phi_1\in H^{\frac12}(\d\Omega)$ be given in \eqref{Phi0} in terms of $u_0$ and some fixed constant $C_0\in\R$.

If $\Phi\in H^2(\Omega)$ is a weak solution of the boundary value problem \eqref{BVP:Phi},
then   the pair of  Frolov's form \eqref{Fro}:  
$(\rho, u)=\bigl(\eta(\Phi), \,\nabla^\perp \Phi\bigr)$ 
 is a weak solution of the boundary value problem \eqref{BVP}.
 
\end{lemma}
 
\subsubsection{Existence and regularity results}\label{subs:ex,re}
Our main   theorem concerning the existence and the regularity properties of the weak solutions to the Navier-Stokes system \eqref{SNS} as well as to the elliptic equation \eqref{elliptic} reads as follows.  
\begin{theorem}[Existence and regularity results]\label{thm}
Let $\eta\in L^\infty(\R;[0,\infty))$,
 $b\in C(\R;[\mu_\ast, \mu^\ast])$, $0<\mu_\ast\leq \mu^\ast$ be given.
 Let $\Omega\subset\R^2$ be a bounded simply connected $C^{1,1}$ domain.
Let $f\in H^{-1}(\Omega;\R^2)$ be given.
\begin{enumerate}[(i)]
\item     Then for any $\Phi_0\in H^{\frac32}(\partial\Omega), \Phi_1\in H^{\frac12}(\d\Omega)$,  there exists at least one weak solution   $\Phi\in H^2(\Omega)$  of the boundary value problem \eqref{BVP:Phi}.

\item
Let $C_0\in\R$ and $u_0\in H^{\frac12}(\partial\Omega; \R^2)$ with $\int_{\d\Omega}u_0\cdot n=0$. 
If $\Phi_0\in H^{\frac32}(\partial\Omega), \Phi_1\in H^{\frac12}(\d\Omega)$ are given by \eqref{Phi0} and $\Phi\in H^2(\Omega)$  is a weak solution of \eqref{BVP:Phi}, then the pair of Frolov's form
\begin{equation}\label{pair}
(\rho, u )=\bigl(\eta(\Phi), \,\nabla^\perp \Phi \bigr)
\end{equation}
  is a weak solution of the boundary value problem \eqref{BVP} with $u\in H^1(\Omega; \R^2)$. 


\end{enumerate} 

\medskip

Furthermore, we have the following regularity results under additional smoothness assumptions.
\begin{enumerate}[(1)]
\item If $\Omega$ is a  connected $C^{2,1}$ domain, 
 the function $\eta$ is taken to be continuous and $f\in L^2(\Omega;\R^2)$,
 then for any $ \Phi_0\in H^{\frac52}(\d\Omega), \Phi_1\in H^{\frac32}(\d\Omega)$ (resp. $u_0\in H^{\frac32}(\partial\Omega; \R^2)$) the weak solution $\Phi$ (resp. $u$) given in (i) (resp. (ii)) belongs to $W^{2,p}(\Omega)$ (resp.   $W^{1,p}(\Omega; \R^2)$), for all $p\in [1,\infty)$. 

\item Let $k\geq 2$ be an integer. If $\Omega$ is a connected $C^{k+1,1}$ domain, the functions $\eta, b\in C^k_b(\R)=\{h\in C^k(\R)\,|\,\|h^{(j)}\|_{L^\infty}<\infty,\,\forall 0\leq j\leq k\}$ and $f\in H^{k-1}(\Omega;\R^2)$, then  for any $\Phi_0\in H^{k+\frac32}(\d\Omega), \Phi_1\in H^{k+\frac12}(\d\Omega)$ (resp. $u_0\in H^{k+\frac12}(\d\Omega)$),  the weak solution $\Phi$ (resp. $u$) given in (i) (resp. (ii)) belongs to $W^{k+1,p}(\Omega)$ (resp. $W^{k,p}(\Omega; \R^2)$) for all $1\leq p<\infty$.
In particular, if $k=2$, then $u\in W^{2,p}(\Omega;\R^2)$, $p>2$ is Lipschitz continuous.

\end{enumerate}
\end{theorem}  

Theorem \ref{thm} will be proved in Section \ref{proof}, where we will  follow J. Leray's approach in \cite{Leray} for the resolution of the classical stationary Navier-Stokes equation. 
    By virtue of the above Lemma \ref{fact}, it remains to study  the fourth-order nonlinear elliptic equation \eqref{BVP:Phi} for the stream function $\Phi$.
Compared to the classical case, we here have to  pay more attention on the nonlinear dependence of the density $\rho$ and the viscosity coefficient $\mu$ on  $\Phi$.

For the completeness of the results, we will establish the existence and regularity results in the exterior domain and the whole plane cases in Appendix \ref{app:ex}.      

\begin{remark} 

We can also consider the system \eqref{BVP} in bounded domains of other types, following the arguments for the classical Navier--Stokes equations \eqref{CNS}. 
For example, it is obvious that the existence and regularity results in Theorem \ref{thm} hold true on a bounded multi-connected domain $\cup_{i=1}^n\Omega_i$, under zero flux assumption on the boundary of each connected component \eqref{ui}. 

 The existence result in Theorem \ref{thm} can also be easily extended to the strip domain $\R\times[0,1]$ by use of Poincaré inequality. 

We can follow the idea in \cite{Guillod3} by J. Guillod and P. Wittwer for \eqref{CNS} on the half plane, to show the solvability of \eqref{BVP} on the half plane $\R\times[0,\infty)$ by assuming small boundary value $\|u_0\|_{L^\infty}$  on the  unbounded boundary $\R\times\{0\}$.

\end{remark}


\subsection{Symmetric solutions}\label{subs:symmetry}
We turn to study the stationary Navier-Stokes equations \eqref{SNS} under some symmetry assumptions on the density function in this subsection.

We give first an observation when we write the velocity vector field $u=\nabla^\perp \Phi$ in terms of the stream function $\Phi$.
Let $U\subset\R^2$ be an open set and we consider another coordinate system $(y_1,y_2)$   on it.
We suppose that the Jacobian $\nabla_x y=(\frac{\d y_i}{\d x_j})_{1\leq i,j\leq 2}$ is not degenerate  and we consider the stationary Navier-Stokes system \eqref{SNS} on $U$. 
If   the density function depends only on $y_1$
$$
\rho=\alpha(y_1),
$$
and $\alpha'\neq 0$ does not vanish, then, by formal calculations, the equation
$$
0=\div(\rho u)=\div(\rho\nabla^\perp\Phi)=\alpha'(\nabla_x y_1\cdot\nabla_x^\perp y_2)\d_{y_2}\Phi=\alpha' \det\bigl(\nabla_x y\bigr)\, \d_{y_2}\Phi
$$
implies that $\Phi=\beta(y_1)$ depends also only on $y_1$ on $U$.
Nevertheless it is not necessary that there exists a function $\eta$ such that $\rho=\eta(\Phi)$.
Similarly, if $\Phi$ depends only on $y_1$
$$
\Phi=\beta(y_1),
$$
and $\beta'\neq 0$ does not vanish, then $\rho=\alpha(y_1)$ depends also only on $y_1$ and $\rho=\eta(\Phi)$ with $\eta=\alpha\circ \beta^{-1}$. 
In this case the pair $(\rho, u)=(\alpha(y_1), \nabla^\perp_x (\beta(y_1)))$ is   of Frolov's form \eqref{pair}.

We formulate the solutions (of Frolov's form) to the stationary Navier-Stokes system \eqref{SNS}  when assuming certain symmetries on the density function in the following theorem.
In particular, the Couette flow between a parallel channel, the concentric flow between concentric rotating circles, and the radial flow (also called the Jeffery-Hamel flow) between two nonparallel converging/diverging lines are described. Some explicit solutions of parallel, concentric and radial flows with piecewise-constant
viscosity coefficients will be given in Appendix \ref{example}.
\begin{theorem}[Formulation for the parallel, concentric and radial flows]\label{thm2} 
If the density function
\begin{align*}
\rho=\rho(x_2)\hbox{ in }\R^2,
\,\,\hbox{ or }\,\,\rho(r)\hbox{ in }\R^2\backslash\{0\},
\,\,\hbox{ or }\,\,\rho(\theta)\hbox{ in }\R^2\backslash\{0\},
\,\,\hbox{ with }\rho'\neq 0,
\end{align*}
where $(r,\theta)$ are polar coordinates in $\R^2$, then  the velocity vector field $u$ of the stationary Navier--Stokes equations \eqref{SNS} reads correspondingly as
\begin{equation}\label{u:symmetry}
u=u_1(x_2)\,e_1\hbox{ in }\R^2,
\,\,\hbox{ or }\,\,
rg(r)\,e_\theta\hbox{ in }\R^2\backslash\{0\},
\,\hbox{ or }\,
\frac{h(\theta)}{r}e_r\hbox{ in }\R^2\backslash\{0\},
\end{equation}
where $e_1=\begin{pmatrix}1\\0\end{pmatrix}$, $e_r=\begin{pmatrix}\frac{x_1}{r}\\\frac{x_2}{r}\end{pmatrix}$, $e_\theta=\begin{pmatrix}\frac{x_2}{r}\\-\frac{x_1}{r}\end{pmatrix}$.

Let  the external force $f=0$ in the system \eqref{SNS}, then  the scalar functions $u_1, g, h$ above satisfy the following three ordinary differential equations of second order respectively  
\begin{align}\label{ODE}
&
 \partial_{x_2}(\mu\partial_{x_2}u_1) =C,\notag
 \\
 &  \d_r( \mu r^3 \d_r g) = -Cr,\\
 &\rho h^2 +\d_\theta(\mu \d_\theta h)+4(\mu h)=C, \notag
\end{align}
where $C\in\R$ can be arbitrarily chosen.
Correspondingly the stream function
$$
\Phi=\Phi(x_2)\,\,\hbox{ or }\,\,\Phi(r)\,\,\hbox{ or }\,\,\Phi(\theta)
$$
satisfies the following elliptic equations of fourth order respectively
\begin{align}\label{eq:Phi}
&\d_{x_2x_2}(\mu \d_{x_2x_2}\Phi)=0,\notag
\\
&\d_{rr}\Bigl(\mu r^3\d_r(\frac1r\d_r\Phi)\Bigr)=-C,
\\
&\d_{\theta\theta}(\mu \d_{\theta\theta}\Phi)
+\d_{\theta}\bigl(\rho(\d_\theta\Phi)^2+4\mu\d_\theta\Phi\bigr)=0.\notag
\end{align}
\end{theorem}
\begin{remark}\label{rmk2}  
In the case $\rho=\rho(x_2)$ or $\rho=\rho(r)$, the velocity vector field $u$  is related only to the viscosity coefficient $\mu$ (while not $\rho$).
 Under some Dirichlet boundary conditions the above ODEs \eqref{ODE} with  given functions $\rho, \mu$ can be solved up to a real constant, and hence there are uncountably many solutions to the corresponding boundary value problems of the system \eqref{SNS}. 
  

\end{remark}
\begin{proof}[Proof of Theorem \ref{thm2}]
We are going to consider the cases $\rho=\rho(x_2)$, $\rho=\rho(r)$ and $\rho=\rho(\theta)$ separately.
{}
We notice that if we take the polar coordinate $(r,\theta)$ on the plane $\R^2$, with
$$(x_1, x_2)=(r\cos \theta, r\sin\theta),$$
then
$$\nabla_x=e_r\d_r-\frac{e_\theta}{r}\d_{\theta},
\quad 
 \nabla^\perp_x=\frac{e_r}{r}\d_{\theta}+e_\theta\d_r,
 \hbox{ with }e_r=\begin{pmatrix}\frac{x_1}{r}\\\frac{x_2}{r}\end{pmatrix},
 \quad e_\theta=\begin{pmatrix}\frac{x_2}{r}\\-\frac{x_1}{r}\end{pmatrix}.$$

\noindent \textbf{Case $\rho=\rho(x_2)$}

If $\rho=\rho(x_2)$ with $\rho'\neq 0$, then the equations $\div(\rho u)=0$ and $\div u=0$ imply that $u_2=0$ and $\partial_{x_1}u_1=0$. Thus $u_1=u_1(x_2)$. Hence
\begin{equation}\label{nonlinearity:x2}
\rho (u\cdot\nabla)u=0\in\R^2,
 \quad  \div(\mu(S u))= 
  \partial_{x_2}(\mu\partial_{x_2}u_1)\,e_1,
  \quad \Delta u=(\d_{x_2x_2}u_1)\,e_1.
  \end{equation}

If $f=0$, then the system \eqref{SNS}   reads as 
$$
\begin{pmatrix}
-\d_{x_2}(\mu\d_{x_2}u_1)+\d_{x_1}\Pi
\\
\d_{x_2}\Pi\end{pmatrix}=\begin{pmatrix}0\\0\end{pmatrix}.
$$
The equation $\d_{x_2}\Pi=0$ implies $\Pi=\Pi(x_1)$.
Thus there exists a constant $C\in \R$ such that
$$ \partial_{x_2}(\mu\partial_{x_2}u_1) = -\d_{x_1}\Pi =C.$$ 

\noindent\textbf{Case $\rho=\rho(r)$}

If $\rho=\rho(r)$ with $\rho'\neq 0$, then the equations $\div(\rho u)=0$ and $\div u=0$ imply that $u\cdot e_r=0$ and hence $u=g_1(r,\theta)e_\theta$ for some scalar function $g_1$.
The incompressibility $\div u=0$ then implies $(\d_r g_1) e_r\cdot e_\theta-(\d_\theta g_1)\frac{e_\theta}{r}\cdot e_\theta=0$, that is, $\d_\theta g_1=0$. 
 Thus $u=g_1(r)e_\theta$.

Let 
$$
g(r)=\frac{g_1(r)}{r},
\hbox{ such that }u=rg(r)e_\theta,
$$
then it is straightforward to calculate
\begin{align*}
&\nabla u=\begin{pmatrix} rg'\frac{x_1x_2}{r^2} & g+rg'\frac{x_2^2}{r^2}\\ -g-rg'\frac{x_1^2}{r^2} & - rg'\frac{x_1x_2}{r^2} \end{pmatrix},
\quad
 S u=\nabla u +\nabla^{T} u=rg'\begin{pmatrix} 2\frac{x_1x_2}{r^2} & \frac{x_2^2-x_1^2}{r^2} \\ \frac{x_2^2-x_1^2}{r^2}& -2\frac{x_1x_2}{r^2} \end{pmatrix},   
\end{align*} 
and
 \begin{equation}\label{nonlinearity}\begin{split}
& \rho (u\cdot\nabla)u=-r\rho g^2 e_r,
 \quad  \div(\mu(S u))= \frac{\d_r(r^3 \mu \d_r g)}{r^2} e_\theta,
 \\
&\Delta u= (r\d_{rr}g+3\d_rg) e_\theta.
\end{split} \end{equation}

If $f=0$, then the system \eqref{SNS}  reads as
\begin{equation}\label{SNS:r}
\bigl(-r\rho g^2+\partial_r\Pi\bigr)e_r
+\bigl( -\frac{\d_r(r^3 \mu \d_r g)}{r^2}-\frac1r\partial_\theta\Pi\bigr)e_\theta=0.
\end{equation}
Since $\mu=\mu(r)$ and $g=g(r)$, we derive from the above equation \eqref{SNS:r} in the $e_\theta$-direction that $\partial_{\theta}\Pi=\alpha(r)$,  where $\alpha$ is a function depending only on $r$. Then $\Pi$ has the form $\Pi(r,\theta)=\alpha(r)\theta+\beta(r)$, where $\beta$ is a function depending only on $r$. 
The above equation \eqref{SNS:r} in the $e_r$-direction  implies that $\partial_r\Pi$ depends only on $r$ and hence $\alpha(r)=C$ is a constant, such that  
\begin{equation*} 
\Pi(r,\theta)=C\theta+\beta(r).
\end{equation*}  
We substitute $\d_{\theta}\Pi=C$ into the equation \eqref{SNS:r} to obtain $\eqref{ODE}_2$. \\

\noindent\textbf{Case $\rho=\rho(\theta)$}

If $\rho=\rho(\theta)$ with $\rho'\neq 0$, then the equations $\div(\rho u)=0$ and $\div u=0$ imply that $u\cdot e_\theta=0$ and hence $u=h_1(r,\theta)e_r$ for some scalar function $h_1$.
The incompressibility $\div u=0$ then implies 
$$\d_r h_1+\frac1r h_1=0.$$  
 Thus $h_1(r,\theta)=\frac{h(\theta)}{r}$ and $u=\frac{h(\theta)}{r}e_r$.
It is straightforward to calculate
\begin{align*}
&\nabla u=\frac1{r^4}\begin{pmatrix} 
-(x_1^2-x_2^2)h- x_1x_2 h' & -2 x_1x_2h+ x_1^2h'\\ 
-2 x_1x_2 h- x_2^2 h' & ( x_1^2-x_2^2) h+ x_1x_2h' 
\end{pmatrix}, 
\\
& S u=\nabla u+\nabla^Tu=\frac{1}{r^4}
 \begin{pmatrix} -2(x_1^2-x_2^2)h-2x_1x_2h'  &-4x_1x_2 h+(x_1^2-x_2^2)h' \\ -4x_1x_2 h+(x_1^2-x_2^2)h' & 2(x_1^2-x_2^2)h+2x_1x_2h'  \end{pmatrix},   
\end{align*}  
and
 \begin{equation}\label{nonlinearity:theta}\begin{split}
& \rho (u\cdot\nabla)u= -\rho\frac{h^2}{r^3} e_r,
 \quad  \div(\mu(S u))= \frac{\d_\theta (\mu \d_\theta h)}{r^3}e_r- 
2\frac{\d_\theta(\mu h)}{r^3}e_\theta,
\\
&  \Delta u= \frac{\d_{\theta\theta}h}{r^3} e_r-2\frac{\d_\theta h}{r^3}e_\theta.
\end{split} \end{equation}
 Thus the system \eqref{SNS} with $f=0$ reads as
\begin{equation}\label{SNS:theta}
\Bigl(-\rho\frac{h^2}{r^3} -\frac{\d_\theta(\mu\d_\theta h)}{r^3}+\partial_r\Pi\Bigr)e_r
+\Bigl( 2\frac{\d_\theta(  \mu h)}{r^3}-\frac1r\partial_\theta\Pi\Bigr)e_\theta=0.
\end{equation}
We derive from the above equation \eqref{SNS:theta} in the $e_\theta$-direction that $\partial_{\theta}\Pi=2r^{-2}\d_\theta(\mu h)$.
Since $\mu=\mu(\theta)$ and $h=h(\theta)$, $\Pi$ has the form
$$
\Pi(r,\theta)=2r^{-2}(\mu h)+\alpha(r),
$$
  where $\alpha$ is a function depending only on $r$. 
  We substitute $\d_r \Pi=-\frac{4}{r^3}(\mu h)+\alpha'(r)$ into \eqref{SNS:theta} to derive
\begin{equation*} 
\rho h^2 +\d_\theta(\mu \d_\theta h)+4(\mu h) = r^3 \alpha'(r), 
\end{equation*} 
where the left-hand side depends only on $\theta$ and the right-hand side depends only on $r$.
Hence there exists $C\in\R$ such that $\eqref{ODE}_3$ holds.

\end{proof}

We have the following  irregularity results, as a straightforward consequence from Theorem \ref{thm2}.  
\begin{corollary}[Irregularity results with piecewise-constant viscosity coefficients]\label{corollary}
For the parallel, concentric and radial flows formulated in Theorem \ref{thm2} above, if we assume that the viscosity coefficient 
\begin{equation}\label{coro:mu}
\mu=\mu(x_2),\,\,\hbox{ or }\,\,\mu(r),\,\,\hbox{ or }\mu(\theta)\,\,\hbox{ is a step function jumping at }
a\in (0,2\pi),
\end{equation}
$\rho, \mu$ have positive lower and upper bounds, and that
\begin{align}\label{coro:assumption}
& \d_{x_2} u_1\in L^1_\loc(\R),\,\,\hbox{ or }\,\,  \d_r g\in L^1_\loc(\R^+),\,\,\hbox{ or }\,\,\, h\hbox{ and }  \d_\theta h\in L^1_\loc([0,2\pi))
\\
&\hbox{do not vanish in a neighborhood }U_a \hbox{ of }a,\notag
\end{align}
then
\begin{align*}
&\Delta u=(\d_{x_2x_2}u_1)\,e_1 \not\in L^1_\loc(\R^2), 
\\
& \hbox{ or }\,\, (r\d_{rr}g+3\d_rg) e_\theta \not\in L^1_\loc(\R^2\backslash\{0\}), 
\\
& \hbox{ or } \,\,  \frac{\d_{\theta\theta}h}{r^3} e_r-2\frac{\d_\theta h}{r^3}e_\theta \not\in L^1_\loc(\R^2\backslash\{0\}).
\end{align*}
In the case of radial flow $(\rho, u)=(\rho(\theta), \frac{h(\theta)}{r}e_r)$,  we also have
$$
 \div(\mu Su)=\frac{\d_\theta (\mu \d_\theta h)}{r^3}e_r
 - 
2\frac{\d_\theta(\mu h)}{r^3}e_\theta\not\in L^1_\loc(\R^2\backslash\{0\}).
$$ 
\end{corollary}
\begin{proof}
If the viscosity coefficient $\mu=\mu(x_2)$ or $\mu(r)$ or $\mu(\theta)$ is a step function with the jump point at $a$, then $\mu'$ is the delta distribution $\delta_a$ (up to a constant) which does not belong to $L^1(U_a)$, with $U_a$ a neighborhood of $a$.
The expressions for $\Delta u$, $\div(\mu Su)$ in Corollary \ref{corollary} can be found in \eqref{nonlinearity:x2}, \eqref{nonlinearity} and \eqref{nonlinearity:theta} above.
  
  We assume by contradiction that 
\begin{align*}
 & \Delta u=(\d_{x_2x_2}u_1)e_1\in L^1_\loc(\R^2),
 \\
 &\hbox{ or }\,\,
  (r\d_{rr}g+3\d_rg) e_\theta  \in L^1_\loc(\R^2\backslash\{0\}),
  \\
  &\hbox{ or }\,\,
  \frac{\d_{\theta\theta}h}{r^3} e_r-2\frac{\d_\theta h}{r^3}e_\theta  \in L^1_\loc(\R^2\backslash\{0\}),
  \end{align*}
then by the assumptions \eqref{coro:assumption} we have
 \begin{align*}
&\d_{x_2}u_1\in W^{1,1}_\loc(\R)\subset L^\infty_\loc(\R),
  \\
  &\hbox{ or }\,\,
  \d_r g\in W^{1,1}_\loc(\R^+)\subset L^\infty_\loc(\R^+),
   \\
   &\hbox{ or }\,\,
    h, \d_\theta h\in W^{1,1}_\loc ([0,2\pi))\subset L^\infty_\loc([0,2\pi)).
\end{align*}
  Thus by the ODEs \eqref{ODE} and the assumption \eqref{coro:assumption}, in the neighborhood $U_a$,
  \begin{align*}
  &\d_{x_2}\mu=\frac{1}{\d_{x_2}u_1}(C-\mu\d_{x_2x_2}u_1)\in L^1_\loc(U_a),
     \\
     &\hbox{ or }\,\,
      \d_r\mu=\frac{1}{r^3\d_r g}\bigl( -Cr-\mu\d_r(r^3\d_r g)\bigr)\in L^1_\loc(U_a),
      \\
&   \hbox{ or }\,\,
            \d_\theta\mu=\frac{1}{\d_\theta h}\bigl( C-4\mu h-\rho h^2-\mu\d_{\theta\theta}h\bigr)\in L^1_\loc(U_a).
  \end{align*}
  This is a contradiction to \eqref{coro:mu}.
  
Similarly, in the case of radial flow $(\rho, u)=(\rho(\theta), \frac{h(\theta)}{r}e_r)$, if we assume by contradiction that 
$$\div(\mu Su)=\frac{\d_\theta (\mu \d_\theta h)}{r^3}e_r
 - 
2\frac{\d_\theta(\mu h)}{r^3}e_\theta \in L^1_\loc(\R^2\backslash\{0\}),$$
 then by the ODE $\eqref{ODE}_3$ and the assumptions \eqref{coro:assumption} we have
$$
\d_\theta(\mu\d_\theta h)=C-4\mu h-\rho h^2\in L^1_\loc([0,2\pi)),
\hbox{ and hence }\d_\theta(\mu h)\in L^1_\loc([0,2\pi)),
$$
which implies the following which is a contradiction to \eqref{coro:mu}:
$$
\d_\theta\mu=\frac{1}{h}(\d_\theta(\mu h)-\mu \d_\theta h)\in L^1_\loc(U_a).
$$ 

\end{proof}

\begin{remark}\label{rmk:irregulartiy}
It is also straightforward to see that the (first-order) derivative of
\begin{align*}
   & \mu \Delta\Phi=\mu(\d_{x_2}u_1-\d_{x_1}u_2),
   \\
   &\hbox{ or }\mu\d_{x_2x_2}\Phi=\mu\d_{x_2}u_1,
   \\
   &\hbox{ or }\mu\d_{x_1x_1}\Phi=-\mu\d_{x_1}u_2,
   \\
&    \hbox{ or }\mu(\d_{x_2x_2}\Phi-\d_{x_1x_1}\Phi)=\mu(\d_{x_2}u_1+\d_{x_1}u_2),
\\
&    \hbox{ or }\mu(\d_{x_1x_2}\Phi)=\mu(\d_{x_1}u_1)
\end{align*}
is not always locally integrable for piecewise-constant viscosity coefficients.
\end{remark}

We   calculate explicitly some solutions to the Navier-Stokes system \eqref{SNS}  with piecewise-constant viscosity coefficients in Appendix \ref{example}, and we will see that they are indeed   of Frolov's form.

 \subsection{Further  regularity results}\label{subs:discussion}

In contrast  to the irregularity  results for the solutions of the stationary Navier-Stokes system \eqref{SNS} with  piecewise-constant viscosity coefficients (see Corollary \ref{corollary})
$$
\Delta u\not\in L^1_\loc(\R^2\backslash\{0\}),\quad \div(\mu Su)\not \in L^1_\loc(\R^2\backslash\{0\}),
$$
we should have some regularity results for the velocity gradient and the divergence-free part of the viscous term $\div(\mu Su)$
$$
\nabla u, \quad  \PP\div(\mu Su),
$$
where  $\PP$ is  the Leray-Helmholtz projector.
On the \emph{whole plane $\R^2$}, by use of Fourier transform,  any vector-valued tempered distribution $v\in \cS'(\R^2; \R^2)$ can be decomposed  into its div-free and curl-free parts separately
\begin{align*}
&v=\nabla^\perp V_1+\nabla V_2,
\\
&\hbox{with }\nabla^\perp V_1=\nabla^\perp\Delta^{-1}\nabla^\perp\cdot v=\PP v,
\quad \nabla V_2=\nabla\Delta^{-1}\nabla\cdot v= (1-\PP)v,
\end{align*}
and the Leray-Helmholtz projector $\PP$ (as Calder\'on-Zygmund operator) maps $L^p(\R^2; \R^2)$ into itself, for any $p\in (1,\infty)$.
We can also define $\PP$ on $L^p(\Omega; \R^2)$, $1<p<\infty$ where $\Omega$ is a \emph{bounded $C^1$ domain}  and we recall here briefly a possible definition (see \cite{FMM} for more details).
Let $v\in L^p(\Omega; \R^2)$ and let $\PPi_\Omega: \mathcal{E}'(\Omega)\mapsto \mathcal{D}'(\Omega)$ be the Newtonian potential operator which acts component-wise on vector fields.
We define the Leray-Helmholtz projector as follows:
\begin{equation}
\PP v=v-\nabla\div\PPi_\Omega(v)-\nabla V,
\end{equation}
where $V\in W^{1,p}(\Omega)$ solves the following Laplacian equation with Neumann boundary condition 
\begin{equation*}
\left\{\begin{array}{c}
\Delta V=0\hbox{ in }\Omega,
\\
\frac{\partial V}{\partial n}=\bigl( v-\nabla\div\PPi_\Omega(v)\bigr)\cdot n \hbox{ on }\d\Omega.
\end{array}\right.
\end{equation*}
By the results in Section 11 in  \cite{FMM}, we have the following Helmholtz-decomposition 
\begin{equation*}
L^p(\Omega;\R^2)=L^p_{\div, 0}(\Omega) \oplus \hbox{grad}W^{1,p}(\Omega),
\end{equation*}
where 
\begin{align*}
&L^p_{\div, 0}(\Omega)\defeq\{v\in L^p(\Omega; \R^2)\,|\, \div v=0,\, v\cdot n|_{\d\Omega}=0\},
\\
&\hbox{grad} W^{1,p}(\Omega)\defeq \{\nabla V\,|\, V\in W^{1,p}(\Omega)\},
\end{align*}
and the orthogonal Leray-Helmholtz projector $\PP: L^p(\Omega)\mapsto L^p_{\div, 0}(\Omega)$ is bounded and onto.

In this subsection we always consider the stationary Navier-Stokes system \eqref{SNS}  on a bounded $C^{1,1}$ domain $\Omega$, with zero external force $f=0$  (noticing $\div(\rho u\otimes u)=\rho u\cdot\nabla u$ by the density equation $\div (\rho u)=0$)
\begin{equation}\label{SNS0}
\left\{
\begin{aligned}
&\rho u\cdot\nabla u-\di(\mu Su)+\nabla\Pi= 0,\\
&\di u=0,\,\di(\rho u)=0.
\end{aligned}
\right.
\end{equation} 
We   apply the Leray-Helmholtz projector $\PP$ to the first equation of the stationary Navier-Stokes system \eqref{SNS0}  to derive  (whenever one side is well-defined)
\begin{equation}\label{PP}
\PP\div(\mu Su)=\PP(\rho u\cdot\nabla u).
\end{equation}

 We observe  also the  following (formal) one-to-one correspondence between $L_\mu\Phi$ and $\PP\di(\mu Su)$: 
  $$ L_\mu \Phi=\nabla^\perp\cdot \PP\di(\mu Su),
 \quad \PP\di(\mu Su)=\nabla^\perp \Delta^{-1} L_\mu \Phi,
$$  
where $\Phi$ denotes the stream function and we calculated straightforward in \eqref{PP,Lmu} that
\begin{align*}
    L_\mu\Phi:=(\d_{x_2x_2}-\d_{x_1x_1})\bigl( \mu (\d_{x_2x_2}-\d_{x_1x_1})\Phi\bigr)+2\d_{x_1x_2}(\mu 2\d_{x_1x_2}\Phi)=\nabla^\perp\cdot\div(\mu Su).
\end{align*} 
The equality $L_\mu\Phi=\nabla^\perp\cdot(\rho u\cdot\nabla u)$ will help to derive the regularity results below.

\subsubsection{Regularity results}
We have the following regularity results under some (partial) regularity assumptions on the viscosity coefficients (e.g. in the case of pieceweise-constant viscosity coefficients).
\begin{theorem}[$L^p$-boundedness for $\nabla u$ and  $\PP\div(\mu Su)$]\label{thm3} 
Let $\Omega$ be a  bounded $C^{1,1}$ domain.
For any weak solution $(\rho, u)\in L^\infty(\Omega)\times H^1(\Omega; \R^2)$ to the boundary value problem of the  stationary Navier-Stokes system \eqref{BVP} with zero external force $f=0$   (e.g. the solutions given in Theorem \ref{thm}), we have 
\begin{equation}\label{PP:12}
 \PP\div(\mu Su)\in L^p(\Omega; \R^2), \quad \forall p\in (1,2).
\end{equation}

 If furthermore the boundary value $u_0\in W^{1,\infty}(\d\Omega)$  and the viscosity coefficient $\mu\in [\mu_\ast, \mu^\ast]$, $\mu_\ast, \mu^\ast>0$ is a \emph{variably partially BMO coefficient} (e.g. striated pieceweise-constant viscosity coefficients as in Theorem \ref{thm2}), i.e.  there exist $R_{0} \in(0,1]$ and  $\gamma=\gamma(p, \mu_\ast,\mu^\ast)\in(0,1/20)$ such that for any $x \in \Omega$ and any $r \in (0, \min \{R_{0}, \operatorname{dist}(x, \partial \Omega) / 2\} )$ there exists a coordinate system $(y_1,y_2)$ depending on $x$ and $r$
such that 
\begin{equation*} 
	\frac{1}{|B_r(x)|}\int_{B_{r}( x)}\biggl|\mu( y_1,y_2)-\frac{1}{2r}\int_{y_2-r}^{y_2+r} \mu(y_{1}, s) \,d s\biggr| \,d y  \leq \gamma,
\end{equation*} 
then we have 
\begin{equation}\label{PP:2}
\nabla u\in L^p(\Omega; \R^4)\hbox{ and } \PP\div(\mu Su)\in L^p(\Omega; \R^2),\quad \forall p\in [2,\infty).
\end{equation}
\end{theorem}  
\begin{proof}
By Sobolev embedding and H\"older's inequality, we have for the  weak solutions 
$(\rho, u)\in L^\infty(\Omega)\times H^1(\Omega; \R^2)$ that
\begin{align*}
&u\in L^s(\Omega; \R^2)\hbox{ for any }s\in [1,\infty),\\
&\hbox{and hence }\rho u\cdot\nabla u\in L^p(\Omega;\R^2)\hbox{ for any }p\in [1,2).    
\end{align*}
By the $L^p$-estimate for the Leray-Helmholtz projector $\PP$ and the equality \eqref{PP}, we have \eqref{PP:12}.

Recall the fourth-order elliptic equation \eqref{elliptic} for the stream function  $ \Phi$ (we assume $f=0$)
\begin{equation}\label{elliptic:G}
L_\mu\Phi=\nabla^\perp\cdot\div(\rho u\otimes u),
\end{equation}
with $L_\mu =(\d_{x_2x_2}-\d_{x_1x_1}) \mu (\d_{x_2x_2}-\d_{x_1x_1}) +(2\d_{x_1x_2})\mu (2\d_{x_1x_2} )$.
By Sobolev embedding and H\"older's inequality again, for any solutions   $(\rho, u)\in L^\infty(\Omega)\times H^1(\Omega; \R^2)$ we have
$$
\rho u\otimes u\in L^p(\Omega; \R^4), \quad \forall p\in [2,\infty).
$$
For any boundary value  $u_0\in W^{1-\frac1p, p}(\d\Omega)$, $p\in [2,\infty)$, we may assume (with an abuse of notations) $\Phi_0\in W^{2,p}(\Omega)$ to be  the extension of the boundary value defined in \eqref{Phi0} with\footnote{The trace \& inverse trace Theorem \ref{trace} holds also for $u\in W^{1,p}$ and $\Phi\in W^{2,p}$ with $p\in(1,\infty)$, see for example \cite[Theorem 2.4.4]{Galdi}.} $$\|\Phi_0\|_{W^{2,p}}(\Omega)\lesssim \|u_0\|_{W^{1-\frac1p,p}(\d\Omega)}. $$ 
    
By the $L^p$-Estimate for the fourth-order elliptic equation with variably partially BMO coefficient in Theorem 8.6 in \cite{Dong}, we have
\begin{equation*}
\|\Phi\|_{W^{2,p}(\Omega)}\leq C(p,\mu_\ast, \mu^\ast, R_0, |\Omega|)\bigl( \|\rho u\otimes u\|_{L^p(\Omega)} +\|\Phi_0\|_{W^{2,p}(\Omega)}\bigr), 
\end{equation*}
for all $p\in [2,\infty)$.
Thus
$$
\Phi\in W^{2,p}(\Omega)\hbox{ and hence }u\in W^{1,p}(\Omega)\subset L^\infty(\Omega),\quad \forall p\in [2,\infty).
$$
Thus \eqref{PP:2} follows from the equation \eqref{PP}.
\end{proof}

 \begin{remark}[Symmetric flows in Theorem \ref{thm2} revisited]\label{rmk:PP}
Notice that in the parallel flow case  and in the concentric flow case, we have
 $$\PP\di(\mu Su)=\di(\mu Su),$$
 which is smooth by view of \eqref{ODE}, \eqref{nonlinearity:x2} and \eqref{nonlinearity}. 
  
 In the radial flow case $\rho=\rho(\theta)$, we assume  that 
 $$\di(\mu Su)=\nabla^\perp(\frac{\alpha(\theta)}{r^2})+\nabla(\frac{\beta(\theta)}{r^2}),$$
  where $\alpha=\alpha(\theta)$, $\beta=\beta(\theta)$ are scalar functions depending only on $\theta$. 
  Then by \eqref{nonlinearity:theta} and $\eqref{ODE}_3$: $\rho h^2 + (\mu h')'+4(\mu h)=C$, $\alpha$ and $\beta$ satisfy
\begin{equation*}
\left\{
\begin{aligned}
&-2\beta+\alpha'=(\mu h')'\\ &\beta'+2\alpha=2(\mu h)',
\end{aligned}
\right.
\end{equation*} 
that is,
\begin{equation*}
\left\{
\begin{aligned}
&\alpha''+4\alpha=4(\mu h)'+(\mu h')''=-(\rho h^2)',\\
&\beta'' +4\beta=-2(\mu h')'+2(\mu h)''=2(\rho h^2+4\mu h-C)+2(\mu h)''.
\end{aligned}
\right.
\end{equation*} 
We calculate straightforward (in the sense of distribution) that
\begin{align*}
\alpha
&=-\frac{\sin({2\theta})}{2}\int_0^\theta \cos({2s})(\rho h^2)'(s)\,ds
+\frac{\cos({2\theta})}{2}\int_0^{\theta} \sin(2s)(\rho h^2)'(s)\,ds
\\
&\quad +C_1\sin({2\theta})+C_2 \cos({2\theta}),
\\
\d_\theta\alpha
&=-\cos({2\theta})\int_0^\theta \cos({2s})(\rho h^2)'(s)\,ds
- \sin({2\theta})\int_0^{\theta} \sin(2s)(\rho h^2)'(s)\,ds
\\
&\quad+2C_1\cos({2\theta})-2C_2 \sin({2\theta}),
\end{align*}
for some real constants $C_1, C_1\in\R$.
It is then easy to see that if $\rho h^2\in L^p([0,2\pi))$ then $\alpha, \d_\theta\alpha\in L^p([0,2\pi))$ and hence
 $$\PP\di(\mu Su)=\nabla^\perp(\frac{\alpha}{r^2})=-\frac{2\alpha}{r^3}e_\theta+\frac{\d_\theta \alpha}{r^3} e_r\in L^p_\loc(\R^2\backslash\{0\}).$$
 However, if $\mu$ has a jump at the point $a$, then $\beta$ also has a jump at $a$, since (formally) $\beta'=2(\mu h)'-2\alpha=2\mu h'+2\mu' h-2\alpha$. Thus $\div(\mu Su)\not\in L^1_\loc(\R^2\backslash\{0\})$, which has been claimed in Corollary \ref{corollary}.

\end{remark}

\subsubsection{Further discussions on the  fourth-order elliptic equation}
We conclude this section with some further discussions on the fourth-order elliptic equation \eqref{elliptic}.
  If the right-hand of the equation \eqref{elliptic} simply vanishes, that is, 
\begin{equation}\label{beltrami2}
L_\mu\Phi=(\d_{x_2x_2}-\d_{x_1x_1}) \mu (\d_{x_2x_2}-\d_{x_1x_1})\Phi +(2\d_{x_1x_2})\mu (2\d_{x_1x_2} )\Phi=0,
\end{equation}
then with the function $\Psi:\R^2\to \R$ satisfying  
$$ \begin{pmatrix}
    \mu (\d_{22}-\d_{11})\Phi \\\mu 2\d_{12}\Phi
    \end{pmatrix}= \begin{pmatrix}
      -2\d_{12}\Psi\\ (\d_{22}-\d_{11})\Psi
    \end{pmatrix},$$
  the complex value function $\Lambda=\Phi+i\Psi$ solves the following second-order Beltrami-type equation
 \begin{align*}
   \d_{\bar z}^2\Lambda=\frac{1-\mu}{1+\mu}\overline{\d_z^2 \Lambda},
   \quad z=x_1+ix_2.
  \end{align*}
  
  This description can be compared with the first-order Beltrami equation
$$\d_{\bar z}\tilde w=\frac{1-\sigma}{1+\sigma}\overline{\d_z \tilde w}.$$ 
Here $\tilde w=\tilde u+i\tilde v$ is a complex value function, where the real part $\tilde u$ satisfies a second-order  elliptic equation of divergence form
\begin{equation}\label{2ndelliptic}
\di(\sigma(x)\nabla \tilde u)=0,
\end{equation}
and the imaginary part $\tilde v$ is related by $\sigma(x)\nabla \tilde u=\nabla^\perp \tilde v$.
According to \cite{AFS08}, on a bounded domain $\Omega\subset\R^2$, there exists a measurable function $\sigma: \Omega\mapsto \{\frac1K, K\}$, $K>1$ such that the solutions $\tilde u\in H^{1}(\Omega)$ to the equation \eqref{2ndelliptic} with the boundary condition $\tilde u|_{\partial\Omega}=x_1$ satisfies
$$
\int_{B} |\nabla \tilde u|^{\frac{2K}{K-1}}=\infty,
$$
for any disk $B\subset\Omega$. That is, $\tilde u\not\in W^{1,p}(\Omega)$ for any $p\geq \frac{2K}{K-1}$. 

Following the convex integration method in \cite{AFS08} (and after personal discussions with Laszlo Szekelyhidi), one can show that there exists a measurable function $\mu: \Omega\mapsto \{\frac1K, K\}$, $K>1$ such that the solutions $\Phi\in H^2(\Omega)$ of the equation \eqref{beltrami2}  satisfies 
$$
\int_{B} |\nabla^2 \Phi|^{\frac{2K}{K-1}}=\infty,
$$
for any disk $B\subset\Omega$.
Although it is not clear whether this constructed solution $(\rho, u)=(b^{-1}(\mu), \nabla^\perp\Phi)$ solves   the stationary Navier-Stokes equation \eqref{SNS}, we expect in general that the solutions for \eqref{SNS} with only bounded   viscosity coefficient $\mu$ (without any smoothness assumption) 
$$
\nabla u\not\in L^p(\Omega), \hbox{ for any }p\geq p_\ast,
$$
 where    $p_\ast<\infty$ depends on the deviation  $|\mu-1|$.

\section{Proof of Theorem \ref{thm}}\label{proof}
In this section we are going to prove Theorem \ref{thm}.

By virtue of Lemma \ref{fact} and the definitions in Subsection \ref{subs:result}, in order to prove (i) in Theorem \ref{thm}, it suffices to show the existence of the weak solutions  $\Phi\in H^2(\Omega)$ of the boundary value problem  \eqref{BVP:Phi}  
\begin{equation}
\label{ellipticbv}
\left\{
\begin{aligned}
&L_\mu\Phi=\nabla^\perp\cdot\div(\rho \nabla^\perp\Phi\otimes\nabla^\perp\Phi)-\nabla^\perp\cdot f,\\
& \rho=\eta(\Phi),\quad\mu=(b\circ \eta)(\Phi),\\
& \Phi|_{{\partial\Omega}} =\Phi_0,\quad 
 \frac{\d \Phi}{\d n}\big|_{\d\Omega} =\Phi_1,
\end{aligned}
\right.
\end{equation}
where  $L_\mu$ denotes the following fourth-order elliptic operator 
 $$L_\mu=(\d_{x_2x_2}-\d_{x_1x_1}) \mu (\d_{x_2x_2}-\d_{x_1x_1}) +(2\d_{x_1x_2})\mu (2\d_{x_1x_2} ).$$
Here the functions $\eta\in L^\infty(\R;[0,\rho^\ast])$, $0<\rho^\ast$, $b\in C(\R;[\mu_\ast, \mu^\ast])$, $0<\mu_\ast\leq \mu^\ast$ and the data $\Phi_0\in H^{\frac32}(\d\Omega)$, $\Phi_1\in H^{\frac12}(\d\Omega)$, $f\in H^{-1}(\Omega;\R^2)$ are given.
  

We will first show the solvability of \eqref{ellipticbv} on a bounded simply connected $C^{1,1}$ domain in Subsection \ref{subs:bdd}, 
and then the regularity results in Theorem \ref{thm}   in Subsection \ref{subs:k}.

 \subsection{Solvability on the bounded domain}\label{subs:bdd}
 Let $\Omega$ be a bounded simply connected $C^{1,1}$ domain on $\R^2$.
 
We first treat the boundary condition $\Phi|_{\partial\Omega}=\Phi_0\in H^{\frac32}(\partial\Omega)$ and $\frac{\d \Phi}{\d n}\big|_{\d\Omega} =\Phi_1\in H^{\frac12}(\partial\Omega)$.  By   the inverse trace Theorem~\ref{trace} and the Whitney's extension Theorem, 
we extend $\Phi_0$ on the whole plane $\R^2$  (still denoted by $\Phi_0$) such that $\frac{\d \Phi_0}{\d n}\big|_{\d\Omega} =\Phi_1$ and $\Phi_0\in H^2(\R^2)$.
We then take a sequence of truncated functions $\zeta(x;\delta)$ near the boundary $\partial\Omega$ and define
\begin{equation}\label{Phi0delta}
\Phi_0^\delta(x)=\Phi_0(x)\zeta(x;\delta)\in H^2(\R^2).
\end{equation}
Here  $\zeta(x;\delta) \in C_0^\infty(\R^2)$ is a smooth function, with $\zeta(x;\delta)=1$ near $\partial\Omega$ and $\zeta(x;\delta)=0$ if $\textrm{dist}(x,\partial\Omega)\geq \delta$, such that
    $$|\zeta(x;\delta)|\le C,\quad|\nabla \zeta(x;\delta)|\le C\delta^{-1},\quad\forall\delta\in(0,\delta_1],$$
   for some fixed constants $C>0$ and $\delta_1>0$. 
   Then
   $$
   \Phi_0^\delta|_{\partial\Omega}=\Phi_0|_{\partial\Omega},\quad \frac{\d\Phi^\delta_0}{\d n}|_{\partial\Omega}=\Phi_1.
   $$ 
    
    Fix $\delta>0$.
  If $\Phi\in H^2(\Omega)$ is a weak solution of the elliptic problem   \eqref{ellipticbv}, then 
     \begin{equation}\label{varphi}
     \varphi^\delta\defeq \Phi-\Phi_0^\delta\in H_0^2(\Omega)  
     \end{equation}
     satisfies  
\begin{equation}\label{ellipticvarphi}\begin{split}
& \int_{\Omega}\mu
\Bigl( (\d_{22}\varphi^\delta-\d_{11}\varphi^\delta) (\d_{22}\psi-\d_{11}\psi)
+ (2\d_{12}\varphi^\delta)(2\d_{12}\psi)\Bigr)\,dx
\\
&=\int_{\Omega}\rho(\nabla^\perp(\Phi_0^\delta+\varphi^\delta)
\otimes\nabla^\perp(\Phi_0^\delta+\varphi^\delta)):\nabla\nabla^\perp\psi\,dx
+\int_{\Omega}f\cdot \nabla^\perp\psi\,dx 
\\
&\quad 
- \int_{\Omega}\mu\Bigl( (\d_{22}\Phi_0^\delta-\d_{11}\Phi_0^\delta) (\d_{22}\psi-\d_{11}\psi)
+ (2\d_{12}\Phi_0^\delta)(2\d_{12}\psi)\Bigr)\,dx,
\\
& \quad\hbox{with }\rho=\eta(\Phi_0^\delta+\varphi^\delta)\hbox{ and } \mu=b(\rho),\quad \forall\psi\in H_0^2(\Omega),
\end{split}
\end{equation} 
and vice versa. 
We hence search for $\varphi^\delta\in H_0^2(\Omega)$ satisfying \eqref{ellipticvarphi}.

Fix $\tilde\varphi\in H_0^2(\Omega)$ and set 
\begin{equation}\label{tilderhomu}
\tilde\rho^\delta=\eta(\Phi_0^\delta+\tilde\varphi),\quad \tilde\mu^\delta=(b\circ \eta)(\Phi_0^\delta+\tilde\varphi).
\end{equation}
We take a sequence of mollifiers $(\sigma^\varepsilon)_\varepsilon$ on $\R^2$, with $\sigma^\varepsilon=\frac{1}{\varepsilon^2}\sigma(\frac\cdot\varepsilon)$, $\sigma\in C_0^\infty(\R^2)$, $\int_{\R^2}\sigma=1$. We define
\begin{align}\label{Phi0delta:neu}
\Phi_0^{\varepsilon}=\sigma^\varepsilon\ast\Phi_0
\hbox{ and }
\Phi_0^{\delta,\varepsilon}=\Phi_0^{\varepsilon}(x)\zeta(x;\delta)\in H^3(\R^2),
\end{align}
such that for any fixed $\delta<\delta_1$,
\begin{align}
\label{regularization:1}
&\Phi_0^{\delta,\varepsilon}\rightarrow \Phi_0^\delta \hbox{ in }H^2(\R^2),
\end{align}
where $\Phi_0^\delta$ is given in \eqref{Phi0delta}.
We take a sequence of mollifiers $(\phi^\varepsilon)_\varepsilon$ on $\R$, with $\phi^\varepsilon=\frac{1}{\varepsilon}\phi(\frac\cdot\varepsilon)$, $\phi\in C_0^\infty(\R)$, $\int_{\R}\phi=1$.
We regularize $ \tilde\rho^\delta,  \tilde\mu^\delta, f$ (we simply extend $f$ trivially on the whole plane)
\begin{align*}
&f^\varepsilon=\sigma^\varepsilon\ast f\in L^2(\Omega; \R^2),\\
&\eta^\varepsilon=\phi^\varepsilon\ast \eta\in C^2_b(\R;[0,\rho^*]),\quad b^\varepsilon=\phi^\varepsilon\ast b\in C^2_b(\R;[\mu_*,\mu^*]),
\\
& \tilde\rho^{\delta,\varepsilon}=\eta^\varepsilon(\Phi_0^{\delta,\varepsilon}+\tilde\varphi)\leq\rho^\ast ,
\quad \tilde\mu^{\delta,\varepsilon}=b^\varepsilon\bigl(\tilde\rho^{\delta,\varepsilon}\bigr)\in H^2(\Omega;[\mu_\ast,\mu^\ast]),
\end{align*}
such that  
\begin{align}\label{regularization}
&f^\varepsilon\rightarrow f\hbox{ in }H^{-1}(\Omega; \R^2),
\notag\\
&\tilde\rho^{\delta,\varepsilon}\mathop{\rightharpoonup}\limits^\ast \tilde\rho^\delta
\hbox{ and } \tilde\mu^{\delta,\varepsilon}\mathop{\rightharpoonup}\limits^\ast \tilde\mu^\delta
\hbox{ in }L^\infty(\Omega)\hbox{ as }\varepsilon\rightarrow0.
\end{align}

  In the following we are going to find $\varphi^\delta\in H_0^2(\Omega)$ satisfying  \eqref{ellipticvarphi} in three steps. 
In Step 1 we will search for  the unique   $\varphi\in H_0^2(\Omega)$ satisfying
\begin{equation}\label{A}\begin{split}
& \int_{\Omega}\tilde\mu^{\delta,\varepsilon}
\Bigl( (\d_{22}\varphi-\d_{11}\varphi) (\d_{22}\psi-\d_{11}\psi)
+ (2\d_{12}\varphi)(2\d_{12}\psi)\Bigr)\,dx
\\
&=\int_{\Omega}\tilde\rho^{\delta,\varepsilon}(\nabla^\perp(\Phi_0^{\delta,\varepsilon}+\tilde\varphi)
\otimes\nabla^\perp(\Phi_0^{\delta,\varepsilon}+\varphi)):\nabla\nabla^\perp\psi\,dx
+\int_{\Omega}f^\varepsilon\cdot \nabla^\perp\psi\,dx 
\\
& \quad
- \int_{\Omega}\tilde\mu^{\delta,\varepsilon}\Bigl( (\d_{22}\Phi_0^{\delta,\varepsilon}-\d_{11}\Phi_0^{\delta,\varepsilon}) (\d_{22}\psi-\d_{11}\psi)
+ (2\d_{12}\Phi_0^{\delta,\varepsilon})(2\d_{12}\psi)\Bigr)\,dx,
\\
&\,\,\forall\psi\in H_0^2(\Omega).
\end{split}
\end{equation} 
This unique solution will be denoted by $\varphi^{\delta, \varepsilon}$.

Similarly, let $\lambda\in [0,1]$ be a parameter. Then there exists a unique solution $\varphi^{\delta,\varepsilon}_\lambda\in H^2_0(\Omega)$ satisfying 
\begin{equation}\label{lam}\begin{split}
&\int_{\Omega}\tilde\mu^{\delta,\varepsilon}_{\lambda}
\Bigl( (\d_{22}\varphi^{\delta,\varepsilon}_{\lambda}-\d_{11}\varphi^{\delta,\varepsilon}_{\lambda}) (\d_{22}\psi-\d_{11}\psi)
+ (2\d_{12}\varphi^{\delta,\varepsilon}_{\lambda})(2\d_{12}\psi)\Bigr)\,dx
\\
&=\lambda\int_{\Omega}\tilde\rho^{\delta,\varepsilon}_{\lambda}(\nabla^\perp(\lambda\Phi_0^{\delta,\varepsilon}+\tilde\varphi)
\otimes\nabla^\perp(\lambda\Phi_0^{\delta,\varepsilon}+\varphi^{\delta,\varepsilon}_{\lambda})):\nabla\nabla^\perp\psi\,dx
+\lambda\int_{\Omega}f^\varepsilon\cdot \nabla^\perp\psi\,dx 
\\
& \quad
-\lambda\int_{\Omega}\tilde\mu^{\delta,\varepsilon}_{\lambda}\Bigl( (\d_{22}\Phi_0^{\delta,\varepsilon}-\d_{11}\Phi_0^{\delta,\varepsilon}) (\d_{22}\psi-\d_{11}\psi)
+ (2\d_{12}\Phi_0^{\delta,\varepsilon})(2\d_{12}\psi)\Bigr)\,dx,
\\
&\quad\hbox{ with }\tilde\rho^{\delta,\varepsilon}_{\lambda}=\eta^\varepsilon(\lambda\Phi_0^{\delta,\varepsilon}+\tilde\varphi)
\hbox{ and }\tilde\mu^{\delta,\varepsilon}=b^\varepsilon(\tilde\rho^{\delta,\varepsilon}_{\lambda}),
\quad \forall\psi\in H_0^2(\Omega).
\end{split}
\end{equation}
 Notice that  \eqref{A} is \eqref{lam} with $\lambda=1$.

 In Step 2 we will define the map
\begin{equation*}
T^{\delta,\varepsilon}:[0,1]\times H^2_0(\Omega)\ni (\lambda,\tilde\varphi)\mapsto\varphi^{\delta,\varepsilon}_\lambda\in H^2_0(\Omega).
\end{equation*}
Notice that, if $\varphi^{\delta,\varepsilon}_\lambda$ satisfies $\varphi^{\delta,\varepsilon}_\lambda=T^{\delta,\varepsilon}(\lambda,\varphi^{\delta,\varepsilon}_\lambda)$, then \eqref{lam} can be seen as an equality for $\Phi^\varepsilon_\lambda:=\lambda \Phi_0^{\delta,\varepsilon}+\varphi^{\delta,\varepsilon}_\lambda$,
\begin{equation}\label{lam,delta} \begin{split}
& \int_{\Omega}\mu^\varepsilon_\lambda
\Bigl( (\d_{x_2x_2}\Phi^\varepsilon_\lambda-\d_{x_1x_1}\Phi^\varepsilon_\lambda) (\d_{x_2x_2}\psi-\d_{x_1x_1}\psi)
+ (2\d_{x_1x_2}\Phi^\varepsilon_\lambda)(2\d_{x_1x_2}\psi)\Bigr)\,dx
\\
&=\lambda\int_{\Omega}\rho^\varepsilon_\lambda(\nabla^\perp\Phi^\varepsilon_\lambda\otimes\nabla^\perp\Phi^\varepsilon_\lambda):\nabla\nabla^\perp\psi\,dx
+\lambda\int_{\Omega}f\cdot \nabla^\perp\psi\,dx,\quad\forall\psi\in H_0^2(\Omega),
\end{split}
\end{equation} 
where $\rho^\varepsilon_\lambda=\eta^\varepsilon(\Phi^\varepsilon_\lambda)$, $\mu^\varepsilon_\lambda=b^\varepsilon(\rho^\varepsilon_\lambda)$.
We observe that $\Phi^\varepsilon_\lambda$ is independent of $\delta$. 
This fact is going to be used to show a uniform bound on the sequence $(\varphi_\lambda^{\delta,\varepsilon})$.

We are going to show that the map $T^{\delta,\varepsilon}$ has a fixed point with $\lambda=1$
(denoted by $\varphi^{\delta,\varepsilon}$) satisfying \eqref{A} with $\tilde\varphi=\varphi^{\delta,\varepsilon}$:
\begin{equation}\label{T}\begin{split}
&\int_{\Omega} \mu^{\delta,\varepsilon}
\Bigl( (\d_{22}\varphi^{\delta,\varepsilon}-\d_{11}\varphi^{\delta,\varepsilon}) (\d_{22}\psi-\d_{11}\psi)
+ (2\d_{12}\varphi^{\delta,\varepsilon})(2\d_{12}\psi)\Bigr)\,dx
\\
&=\int_{\Omega} \rho^{\delta,\varepsilon}(\nabla^\perp(\Phi_0^{\delta,\varepsilon}+\varphi^{\delta,\varepsilon})
\otimes\nabla^\perp(\Phi_0^{\delta,\varepsilon}+\varphi^{\delta,\varepsilon})):\nabla\nabla^\perp\psi\,dx
+\int_{\Omega}f^\varepsilon\cdot \nabla^\perp\psi\,dx 
\\
& \quad
-\int_{\Omega} \mu^{\delta,\varepsilon}\Bigl( (\d_{22}\Phi_0^{\delta,\varepsilon}-\d_{11}\Phi_0^{\delta,\varepsilon}) (\d_{22}\psi-\d_{11}\psi)
+ (2\d_{12}\Phi_0^{\delta,\varepsilon})(2\d_{12}\psi)\Bigr)\,dx,
\\
&\quad\hbox{ with } \rho^{\delta,\varepsilon} =\eta^\varepsilon( \Phi_0^{\delta,\varepsilon}+\varphi^{\delta,\varepsilon})
\hbox{ and } \mu^{\delta,\varepsilon}=b^\varepsilon( \rho^{\delta,\varepsilon}_{\lambda}),\quad\forall\psi\in H_0^2(\Omega).
\end{split}
\end{equation} 
To show the existence of the fixed point, we will apply the following Leray-Schauder's fixed point theorem, after checking   the conditions (LS1), (LS2) and (LS3) one by one.
\begin{theorem}[Leray-Schauder's fixed point theorem, \cite{Maugeri}]\label{thm:LS}
 Let $B$ be a Banach space. If
\begin{enumerate}[(LS1)]
    \item $T(0,u)=0$, for all $u \in B$,\label{condition 1}
      \item $T$ is a compact map from $B \times[0,1]$ to
$B$, \label{condition 2}
    \item The solutions of $u=T(\lambda, u)$ for all $\lambda\in[0,1]$ are uniformly bounded in $B$.\label{condition 3}
\end{enumerate}
Then there exists $u\in B$ such that $u=T(1,u)$. 
\end{theorem}


In Step 3 we will take $\varepsilon\rightarrow 0$ in the sequence $\{\varphi^{\delta,\varepsilon}\}$ such that the limit $\varphi^\delta$ satisfies \eqref{ellipticvarphi}, and hence $\Phi=\Phi_0^\delta+\varphi^\delta$ solves  the boundary value problem  \eqref{ellipticbv} on the bounded domain $\Omega$.

\smallbreak 

\noindent\textbf{Step 1 Unique solvability  of \eqref{A}.}\\ 
Let $\tilde\varphi\in H_0^2(\Omega)$ be given.  We are going to search for $\varphi\in H_0^2(\Omega)$ satisfying \eqref{A} under the following assumptions on the given functions:
   \begin{equation}\label{A:hypothesis}
    \tilde\rho^{\delta,\varepsilon}=\eta^\varepsilon(\Phi_0^{\delta,\varepsilon}+\tilde\varphi)\leq\rho^\ast,\quad\tilde\mu^{\delta,\varepsilon}\in[\mu_\ast,\mu^\ast],
    \quad \Phi_0^{\delta,\varepsilon}\in H^2(\R^2), 
    \quad f^\varepsilon\in H^{-1}(\Omega; \R^2).
    \end{equation}
    For notational simplicity we do not indicate the upper indices $\delta,\varepsilon$ explicitly  in this step.

We  define the inner product $\langle\cdot,\cdot\rangle$ on the Hilbert space $H_0^2(\Omega; \R)$ as follows:
\begin{equation*}
\label{product}
\langle \varphi, \psi\rangle\defeq  \int_{\Omega}\tilde\mu 
\Bigl( (\d_{22}\varphi-\d_{11}\varphi) (\d_{22}\psi-\d_{11}\psi)
+ (2\d_{12}\varphi)(2\d_{12}\psi)\Bigr)\,dx.
\end{equation*}
Then the corresponding norm $\langle\cdot,\cdot\rangle^{\frac12}$ is equivalent to the $H^2$-norm on $H^2_0(\Omega)$.
Indeed,   
\begin{align*}
&\mu_\ast a       \leq  \langle\varphi,\varphi\rangle 
     =\int_{\Omega}\tilde\mu  \Big(
    (\d_{22}\varphi-\d_{11}\varphi)^2
    +(2\d_{12}\varphi)^2\Big)\,dx
  \leq \mu^\ast a,
\end{align*}
where
$$
a\defeq \int_{\Omega} \Big(
    (\d_{22}\varphi-\d_{11}\varphi)^2
    +(2\d_{12}\varphi)^2\Big)\,dx\geq 0.
$$
By integration by parts and approximation arguments, for $\varphi\in H_0^2(\Omega)$ there holds
\begin{align*}
a& =\int_{\Omega} \Bigl(  (\d_{11}\varphi)^2+(\d_{22}\varphi)^2-2\d_{11}\varphi\d_{22}\varphi+(2\d_{12}\varphi)^2 \Bigr)\,dx
\\
&=\int_{\Omega} \Bigl(  (\d_{11}\varphi)^2+(\d_{22}\varphi)^2+2\d_{11}\varphi\d_{22}\varphi \Bigr)\,dx
=\int_\Omega(\d_{11}\varphi+\d_{22}\varphi)^2\,dx
=\|\Delta\varphi\|_{L^2(\Omega)}^2.
\end{align*}
Thus
\begin{align}\label{equivalence}
\sqrt{\mu_\ast}\|\Delta\varphi\|_{L^2(\Omega)}\leq \langle\varphi,\varphi\rangle^{\frac12}\leq\sqrt{\mu^\ast}\|\Delta\varphi\|_{L^2(\Omega)},
\end{align}
and hence by virtue of the equivalence of the norms $\|\Delta\cdot\|_{L^2(\Omega)}\sim \|\cdot\|_{H^2(\Omega)}$ on $H_0^2(\Omega)$, we have the equivalence of the norms 
$$ \langle\cdot,\cdot\rangle^\frac12 \sim\|\cdot\|_{H^2(\Omega)},
\quad\hbox{ on }H_0^2(\Omega).$$

    Notice that the left-hand side of \eqref{A} reads as $\langle \varphi, \psi\rangle$. We are going to show that the right-hand side of \eqref{A} is a linear functional on $H_0^2(\Omega)$
    $$
    l(\psi),
    $$
    which by Lax-Milgram theorem  defines a unique element (denoted by $A\varphi$) in $H_0^2(\Omega)$ such that 
    $$
    l(\psi)=\langle A\varphi, \psi\rangle.
    $$
    Then we will verify the conditions (LS1), (LS2) and (LS3) in Leray-Schauder's fixed point Theorem \ref{thm:LS}  for  the map 
    $$\alpha A: [0,1]\times H_0^2(\Omega)\mapsto H_0^2(\Omega),$$
    to show  the existence of the unique solution for  the equation
    $$
    \varphi=A\varphi
    $$
    and hence \eqref{A}.
    
\noindent\textbf{Definition of the operator $A$.} By virtue of \eqref{A:hypothesis},    the right-hand side of \eqref{A} depends linearly on $\psi$ and can be bounded by
    \begin{align*}
    &\Bigl(\rho^\ast\|\Phi_0+\tilde\varphi\|_{W^{1,4}}\|\Phi_0+\varphi\|_{W^{1,4}}+\|f\|_{H^{-1}}+8\mu^\ast\|\Phi_0\|_{H^2}\Bigr)\|\psi\|_{H^2}
    \\
    &\leq C(\rho^\ast+\mu^\ast+1)(\|\Phi_0\|_{H^2}+\|\tilde \varphi\|_{H^2}+\|f \|_{H^{-1}}) \bigl(1+\|\Phi_0\|_{H^2}+\|\varphi\|_{H^2}\bigr)\|\psi\|_{H^2},
    \end{align*}
    for some constant $C>0$.
    Here we used the Sobolev's inequality
    $$
    \|g\|_{L^4(\Omega)}\leq C\|g\|_{H^1(\Omega)},\quad\forall g\in H^1(\Omega).
    $$
    Hence the right-hand side of \eqref{A} defines a linear functional $l(\psi)$ on $H_0^2(\Omega)$, which defines correspondingly by Lax-Milgram theorem an element (denoted by $A\varphi$) such that $l(\psi)=\langle A\varphi,\psi\rangle$.
    \smallbreak
      \noindent\textbf{Verification  of Condition (LS1).}   If $\alpha=0$, then the map $\alpha A=0$ trivially.
    \smallbreak
         \noindent\textbf{Verification  of Condition (LS2).}
     In order to show the compactness of the  operator $\alpha A$, we take a weak convergent sequence $(\alpha_n,\varphi_n)\subset [0,1]\times H_0^2(\Omega)$. 
   By virtue of the compact embedding $H_0^2(\Omega)\hookrightarrow W^{1,4}(\Omega)$, there exists a  subsequence (still denoted by $(\alpha_n,\varphi_n)$) converging strongly in $[0,1]\times W^{1,4}(\Omega)$, and hence 
   \begin{align*}
   \|\alpha_n A\varphi_n- \alpha_m A\varphi_m\|_{H^2}
   &\leq \sup_{\|\psi\|_{H^2}=1}\big(|\alpha_n\langle A\varphi_n-A\varphi_m, \psi\rangle|+|(\alpha_n-\alpha_m)\langle A\varphi_m,\psi\rangle|\big)
   \\
   &\leq\sup_{\|\psi\|_{H^2}=1}\Bigl|\int_{\Omega}\tilde\rho\bigl( \nabla^\perp(\Phi_0+\tilde\varphi)\otimes\nabla^\perp (\varphi_n-\varphi_m)\bigr):\nabla\nabla^\perp\psi\,dx\Bigr|\\
   &\quad +|\alpha_n-\alpha_m|\|A\varphi_m\|_{H^2}
   \\
   &\leq C(\|\varphi_n-\varphi_m\|_{W^{1,4}}+|\alpha_m-\alpha_n|)\\
   &\quad\times\bigl(\rho^*\|\Phi_0+\tilde\varphi\|_{W^{1,4}}(1+\|\varphi_m\|_{W^{1,4}})+\|f\|_{H^{-1}}+\mu^*\|\Phi_0\|_{H^2}\bigr)\\
   & 
   \rightarrow0\hbox{ as }n,m\rightarrow\infty.
   \end{align*}
 \smallbreak
 \noindent\textbf{Verification of Condition (LS3).}  
The solutions of $\varphi=\alpha A\varphi$ are uniformly bounded in $H_0^2(\Omega)$. Indeed, if $\varphi=\alpha A\varphi\in H_0^2(\Omega)$, then $\langle \varphi, \psi\rangle=\alpha\langle A\varphi, \psi\rangle=\alpha l(\psi)$ for any $\psi\in H_0^2(\Omega)$, and in particular when $\psi=\varphi$,
\begin{align*}
\langle\varphi,\varphi\rangle
&=\alpha\int_{\Omega}\tilde\rho(\nabla^\perp(\Phi_0+\tilde\varphi)
\otimes\nabla^\perp(\Phi_0+\varphi)):\nabla\nabla^\perp\varphi\,dx
+\alpha\int_{\Omega}f \cdot \nabla^\perp\varphi\,dx 
\\
& \quad
-\alpha\int_{\Omega}\tilde\mu\Bigl( (\d_{22}\Phi_0-\d_{11}\Phi_0) (\d_{22}\varphi-\d_{11}\varphi)
+ (2\d_{12}\Phi_0)(2\d_{12}\varphi)\Bigr)\,dx.
\end{align*}
Notice the equality
\begin{align}\label{convection}
\int_{\Omega}\tilde\rho( \nabla^\perp(\Phi_0+\tilde\varphi)
\otimes\nabla^\perp \varphi):\nabla\nabla^\perp\varphi\,dx
&=\int_{\Omega}\tilde \rho\nabla^\perp(\Phi_0+\tilde\varphi)\cdot\nabla\nabla^\perp\varphi\cdot\nabla^\perp\varphi\,dx
\notag\\
&=-\frac12\int_{\Omega}\div\bigl( \tilde \rho\nabla^\perp(\Phi_0+\tilde\varphi)\bigr)
|\nabla^\perp\varphi|^2\,dx=0,
\end{align}
where we used $\tilde\rho=\eta(\Phi_0+\tilde\varphi)$ in the last equality.
We hence derive from $\langle\varphi, \varphi\rangle=\alpha l(\varphi)$ above  and $\|g\|_{L^4(\Omega)}\leq C\|g\|_{H^1(\Omega)}$ that
\begin{equation}\label{bd}
    \langle\varphi,\varphi\rangle\leq C\alpha  (\rho^\ast+1+\mu^\ast)(\|\Phi_0\|_{H^2}+\|\tilde \varphi\|_{H^2}+\|f\|_{H^{-1}})
    \bigl(1+\|\Phi_0\|_{H^2}\bigr) \|\varphi\|_{H^2}.
\end{equation}
Since  the norm $\langle\cdot,\cdot\rangle^{\frac12}\geq \sqrt{\mu_\ast}\|\Delta\cdot\|_{L^2(\Omega)}$ is equivalent to $\|\cdot\|_{H^2(\Omega)}$ on $H_0^2(\Omega)$, there is a uniform bound for all $\varphi\in H_0^2(\Omega)$ such that $\varphi=\alpha A\varphi$, $\alpha\in [0,1]$:
\begin{equation}\label{bound:varphi}
\|\varphi\|_{H^2}\leq C\mu_\ast^{-1}(\rho^\ast+1+\mu^\ast)(\|\Phi_0\|_{H^2}+\|\tilde \varphi\|_{H^2}+\|f\|_{H^{-1}})
    \bigl(1+\|\Phi_0\|_{H^2}\bigr).
\end{equation}

By Leray-Schauder's Theorem \ref{thm:LS}, there exists a solution of $\varphi=A\varphi$ in $H_0^2(\Omega)$. This solution solves \eqref{A}: $\langle\varphi,\psi\rangle=\langle A\varphi, \psi\rangle=l(\psi)$ for all $\psi\in H_0^2(\Omega)$.
This solution is unique.
Indeed, if there exist two solutions $\varphi_1, \varphi_2\in H_0^2(\Omega)$ of \eqref{A}, then their difference $\dot\varphi=\varphi_1-\varphi_2\in H_0^2(\Omega)$ satisfies
\begin{align*}
\langle\dot\varphi,\psi\rangle=\int_{\Omega}\tilde\rho \nabla^\perp(\Phi_0+ \tilde \varphi)\cdot\nabla\nabla^\perp\psi\cdot\nabla\dot\varphi,\quad\forall\psi\in H_0^2(\Omega).
\end{align*} 
Take $\psi=\dot\varphi$, then by the calculation in \eqref{convection} the right-hand side above vanishes and hence $\dot\varphi=0$, i.e., $\varphi_1=\varphi_2$.

  \medskip
  
  \noindent\textbf{Step 2 Solvability of \eqref{T}.}\\ 
  By the procedure in Step 1 above, we can solve \eqref{lam} uniquely for any $\lambda\in [0,1]$, and we denote this unique solution satisfying \eqref{lam} by $\varphi_\lambda^{\delta,\varepsilon}$.

We are going to check the conditions (LS1), (LS2) and (LS3) for the map $T^{\delta,\varepsilon}:(\lambda,\tilde\varphi)\mapsto \varphi_\lambda^{\delta,\varepsilon}$, in order to show   the existence of the fixed point of $T^{\delta,\varepsilon} $ with $\lambda=1$ by the Leray-Schauder fixed point Theorem \ref{thm:LS}. 
  \smallbreak 
  \noindent\textbf{Verification of Condition (LS1). }Let $\lambda=0$ in \eqref{lam} and let $\varphi^{\delta,\varepsilon}_0$ satisfy $\varphi^{\delta,\varepsilon}_0=T^{\delta,\varepsilon}(0,\tilde\varphi)$. We take $\psi=\varphi^{\delta,\varepsilon}_0$ in \eqref{lam}, which implies
$$\|\Delta \varphi^{\delta,\varepsilon}_0 \|_{L^2}=0.$$
Since $\varphi_0^{\delta,\varepsilon}\in H^2_0(\Omega)$, $\varphi_0^{\delta,\varepsilon}=0$.

\smallbreak
\noindent\textbf{Verification of Condition (LS2).} 
The map
  $$
T^{\delta,\varepsilon}: [0,1]\times H_0^2(\Omega)\ni (\lambda,\tilde\varphi)\mapsto\varphi^{\delta,\varepsilon}_\lambda\in H_0^2(\Omega)
  $$
  is compact, where $\varphi^{\delta,\varepsilon}_\lambda$ is the solution of \eqref{lam}, under the following assumptions on the regularized data:
\begin{equation*}
\Phi_0^{\delta,\varepsilon}\in H^3(\R^2),
\quad \tilde\rho_\lambda^{\delta,\varepsilon}=\eta^\varepsilon(\lambda\Phi_0^{\delta,\varepsilon}+\tilde\varphi)\leq\rho^\ast,
\quad \tilde\mu_\lambda^{\delta,\varepsilon}\in H^2(\Omega),
\quad f^\varepsilon\in L^2(\Omega; \R^2).
\end{equation*}

Indeed, let $(\lambda_n,\tilde\varphi_n)$ be a bounded sequence in $[0,1]\times H^2_0(\Omega)$. Then there exists a   subsequence (still denote by $(\lambda_n,\tilde\varphi_n)$), such that
$$|\lambda_m-\lambda_n|\to 0,\quad \|\tilde\varphi_m-\tilde\varphi_n \|_{W^{1,4}}\to0,\hbox{ as }m,n\to\infty.$$
We denote  $\varphi^{\delta,\varepsilon}_{n}=T^{\delta,\varepsilon}(\lambda_n,\tilde\varphi_n)$, $\tilde\rho^{\delta,\varepsilon}_n=\eta^\varepsilon(\lambda_n\Phi_0^{\delta,\varepsilon}+\tilde\varphi_n)$ and $\tilde\mu^{\delta,\varepsilon}_n= b^\varepsilon(\tilde\rho^{\delta,\varepsilon}_n)$.
We  take the difference between $\eqref{lam}$ with $(\lambda_m,\tilde\varphi_m)$ and $\eqref{lam}$ with $(\lambda_n,\tilde\varphi_n)$. 
Let $\psi=\dot{\varphi}^{\delta,\varepsilon}=\varphi^{\delta,\varepsilon}_m-\varphi^{\delta,\varepsilon}_n$, then (noticing \eqref{convection} again and $\|\cdot\|_{L^\infty(\Omega)}\lesssim \|\cdot\|_{W^{1,4}(\Omega)}$)
{\small\begin{align}\label{differencevarphi}
&\|\Delta(\varphi^{\delta,\varepsilon}_m-\varphi^{\delta,\varepsilon}_n)\|_{L^2}^2\\
&\leq C (\mu_\ast)
\Bigl(  \Bigl|\int_\Omega(\tilde\mu^{\delta,\varepsilon}_m-\tilde\mu^{\delta,\varepsilon}_n)
\Bigl( (\d_{22}\varphi^{\delta,\varepsilon}_m-\d_{11}\varphi^{\delta,\varepsilon}_m) (\d_{22}\dot\varphi^{\delta,\varepsilon}-\d_{11}\dot\varphi^{\delta,\varepsilon})
+ (2\d_{12}\varphi^{\delta,\varepsilon}_m)(2\d_{12}\dot\varphi^{\delta,\varepsilon})\Bigr)\,dx\Bigr|
\notag\\
&\quad +|\lambda_m-\lambda_n|\Bigl|\int_\Omega\tilde\rho^{\delta,\varepsilon}_m(\nabla^\perp(\lambda_m\Phi_0^{\delta,\varepsilon}+\tilde\varphi_m)\otimes\nabla^\perp(\lambda_m\Phi_0^{\delta,\varepsilon}+\varphi^{\delta,\varepsilon}_m)):\nabla\nabla^\perp\dot\varphi^{\delta,\varepsilon} \,dx\Bigr|\notag\\
&\quad +\Bigl| \int_{\Omega}\lambda_n(\tilde\rho^{\delta,\varepsilon}_m-\tilde\rho^{\delta,\varepsilon}_n)(\nabla^\perp(\lambda_m\Phi_0^{\delta,\varepsilon}+\tilde\varphi_m)
\otimes\nabla^\perp(\lambda_m\Phi_0^{\delta,\varepsilon}+\varphi^{\delta,\varepsilon}_m)):\nabla\nabla^\perp\dot\varphi^{\delta,\varepsilon} \,dx\Bigr|\notag\\
&\quad +\Bigl| \int_{\Omega}\lambda_n\tilde\rho^{\delta,\varepsilon}_{n}(\nabla^\perp((\lambda_m-\lambda_n)\Phi_0^{\delta,\varepsilon})
\otimes\nabla^\perp(\lambda_m\Phi_0^{\delta,\varepsilon}+\varphi^{\delta,\varepsilon}_{m})):\nabla\nabla^\perp\dot\varphi^{\delta,\varepsilon} \,dx\Bigr|\notag\\
&\quad +\Bigl| \int_{\Omega}\lambda_n\tilde\rho^{\delta,\varepsilon}_n(\nabla^\perp(\tilde\varphi_m-\tilde\varphi_n)
\otimes\nabla^\perp(\lambda_m\Phi_0^{\delta,\varepsilon}+\varphi^{\delta,\varepsilon}_m)):\nabla\nabla^\perp\dot\varphi^{\delta,\varepsilon} \,dx\Bigr|\notag\\
&\quad +\Bigl| \int_{\Omega}\lambda_n\tilde\rho^{\delta,\varepsilon}_{n}(\nabla^\perp(\lambda_n\Phi_0^{\delta,\varepsilon}+\tilde\varphi_n)
\otimes\nabla^\perp((\lambda_m-\lambda_n)\Phi_0^{\delta,\varepsilon})):\nabla\nabla^\perp\dot\varphi^{\delta,\varepsilon} \,dx\Bigr|\notag\\
&\quad+\Bigl| \int_{\Omega}(\tilde\mu^{\delta,\varepsilon}_m-\tilde\mu^{\delta,\varepsilon}_n)\Bigl( (\d_{22}\Phi_0^{\delta,\varepsilon}-\d_{11}\Phi_0^{\delta,\varepsilon}) (\d_{22}\dot\varphi^{\delta,\varepsilon}-\d_{11}\dot\varphi^{\delta,\varepsilon})
+ (2\d_{12}\Phi_0^{\delta,\varepsilon})(2\d_{12}\dot\varphi^{\delta,\varepsilon})\Bigr)\,dx\Bigr| \Bigr)\notag\\
&\le C(\rho^\ast, \mu_\ast, \mu^\ast)(|\lambda_m-\lambda_n|+\|\tilde\varphi_m-\tilde\varphi_n\|_{W^{1,4}})(1+\|b^\varepsilon(\eta^\varepsilon)\|_{W^{1,\infty}}+\|\eta^\varepsilon\|_{W^{1,\infty}})\notag\\
&\quad\times(1+\|\Delta \varphi^{\delta,\varepsilon}_m\|_{L^2}+\|\Delta \Phi_0^{\delta,\varepsilon}\|_{L^2}) (1+\|\Phi^{\delta,\varepsilon}_0\|_{W^{1,4}}+\|\tilde\varphi_m\|_{W^{1,4}}+\|\tilde\varphi_n\|_{W^{1,4}})\notag\\
&\quad\times(1+\|\Phi^{\delta,\varepsilon}_0\|_{W^{1,4}}+\|\varphi^{\delta,\varepsilon}_m\|_{W^{1,4}})\|\Delta \dot\varphi^{\delta,\varepsilon}\|_{L^2},\notag
\end{align}} 
 Notice that, since $\{\tilde \varphi_n\}$ is uniformly bounded in $H^2$,
the uniform bound of $\{\|\Delta \varphi_n^{\delta,\varepsilon}\|_{L^2}\}$ can be derived similarly to \eqref{bound:varphi}. 
Hence, the following strong convergence holds 
$$\|\Delta(\varphi^{\delta,\varepsilon}_m-\varphi^{\delta,\varepsilon}_n)\|_{L^2}\le C( \|\tilde\varphi_m-\tilde\varphi_n\|_{W^{1,4}}+|\lambda_m-\lambda_n|)\to 0  \hbox{ as } m,\,n\to\infty.$$
The map  $T^{\delta,\varepsilon}: (\lambda,\tilde\varphi)\mapsto\varphi^{\delta,\varepsilon}_\lambda$ is compact.

\smallbreak
\noindent\textbf{Verification of Condition (LS3).}
  Let $\varphi^{\delta,\varepsilon}_{\lambda}$ denote the fixed point of $\varphi= T^{\delta,\varepsilon}(\lambda, \varphi)$ satisfying \eqref{lam}. 
We are going to derive a uniform bound on $\|\varphi^{\delta,\varepsilon}_{\lambda}\|_{H^2}$ by a contradiction argument.
Suppose by contradiction that there exists a subsequence $(\varphi^{\delta,\varepsilon}_{\lambda_n})\subset(\varphi^{\delta,\varepsilon}_{\lambda})$ such that
$$
\|\varphi^{\delta,\varepsilon}_{\lambda_n}\|_{H^2}\rightarrow\infty.
$$
Then we drive from \eqref{lam} with $\psi=\varphi^{\delta,\varepsilon}_{\lambda_n}$ that (noticing again the equality \eqref{convection})
\begin{align}\label{Estimate:varphi,eps}
\mu_\ast\|\Delta \varphi^{\delta,\varepsilon}_{\lambda_n}\|_{L^2}^2
&\leq C(\rho^\ast,\mu_\ast,\mu^\ast)
\Bigl(\bigl( \|\Phi_0^{\delta,\varepsilon}\|_{H^2}^2+\|f^\varepsilon\|_{H^{-1}}+\|\Phi_0^{\delta,\varepsilon}\|_{H^2}\bigr)\|\varphi^{\delta,\varepsilon}_{\lambda_n}\|_{H^2}\notag
\\
&\qquad+\int_{\Omega} \rho^{\delta,\varepsilon}_{\lambda_n}\nabla^\perp \varphi^{\delta,\varepsilon}\cdot\nabla\nabla^\perp\varphi^{\delta,\varepsilon}\cdot
 \nabla^\perp\Phi_0^{\delta,\varepsilon} \,dx\Bigr),
\end{align}
 Let us denote $g^{\delta,\varepsilon}_{\lambda_n}=\frac{\varphi^{\delta,\varepsilon}_{\lambda_n}}{\|\varphi^{\delta,\varepsilon}_{\lambda_n}\|_{H^2}}$, then we drive from the above inequality that 
 \begin{align*}
1\leq \frac{C(\rho^\ast,\mu_\ast,\mu^\ast, \|\Phi_0\|_{H^2},\|f\|_{H^{-1}})}{\|\varphi^{\delta,\varepsilon}_{\lambda_n}\|_{H^2}} 
+C\int_{\Omega}
\Bigl| \nabla^\perp g^{\delta,\varepsilon}_{\lambda_n}\cdot\nabla\nabla^\perp g^{\delta,\varepsilon}_{\lambda_n}\cdot
 \nabla^\perp \Phi_0^{\delta,\varepsilon}\Bigr|\,dx.
\end{align*}
Since $\|g^{\delta,\varepsilon}_{\lambda_n}\|_{H^2}=1$, there exist  subsequences (still denoted by $(g^{\delta,\varepsilon}_{\lambda_n}))$ such that
\begin{align*}
g^{\delta,\varepsilon}_{\lambda_n}\rightharpoonup g^{\varepsilon}  \hbox{ in }H_0^2(\Omega),
\quad g^{\delta,\varepsilon}_{\lambda_n}\rightarrow g^{\varepsilon}  \hbox{ in }W^{1,4}(\Omega).
\end{align*} 
Here the limit $g^{\varepsilon}$ does not depend on $\delta$. Indeed, notice that the $\delta$-independent function $\Phi^\varepsilon_\lambda=\lambda \Phi_0^{\delta,\varepsilon}+\varphi^{\delta,\varepsilon}_\lambda= \lambda \Phi_0^{\delta',\varepsilon}+\varphi^{\delta',\varepsilon}_\lambda$ satisfies \eqref{lam,delta}.
Then 
 $\|\varphi^{\delta',\varepsilon}_{\lambda_n}\|_{H^2}\rightarrow \infty,$
and
\begin{align*}
&\lim_{n\rightarrow\infty}\frac{\varphi^{\delta',\varepsilon}_{\lambda_n}}{\|\varphi^{\delta',\varepsilon}_{\lambda_n}\|_{H^2}}
=\lim_{n\rightarrow\infty}\frac{\varphi^{\delta,\varepsilon}_{\lambda_n}+\lambda\Phi_0^{\delta,\varepsilon}-\lambda\Phi_0^{\delta',\varepsilon}}{\|\varphi^{\delta',\varepsilon}_{\lambda_n}\|_{H^2}}\\
&=\lim_{n\rightarrow\infty}\frac{\varphi^{\delta,\varepsilon}_{\lambda_n}}{\|\varphi^{\delta',\varepsilon}_{\lambda_n}\|_{H^2}}
=\lim_{n\rightarrow\infty}\frac{\varphi^{\delta,\varepsilon}_{\lambda_n}}{\|\varphi^{\delta,\varepsilon}_{\lambda_n}\|_{H^2}}
=g^{\varepsilon}.
\end{align*}

Then taking $n\rightarrow\infty$ in the above inequality we arrive at
\begin{align*}
1\leq  C\int_{\Omega}
\Bigl| \nabla^\perp g^{\varepsilon}\cdot\nabla\nabla^\perp g^{\varepsilon}\cdot
 \nabla^\perp \Phi_0^{\delta,\varepsilon}\Bigr|\,dx.
\end{align*} 
Recall the definition of $\Phi_0^{\delta,\varepsilon}$ in \eqref{Phi0delta:neu} and the regularisation \eqref{regularization:1}, such that
\begin{equation}\label{estimate1}
|\nabla^\perp \Phi_0^{\delta, \varepsilon}|
\leq |\nabla^\perp(\Phi_0^\varepsilon(x)\zeta(x;\delta))|
\leq C(\delta^{-1}|\Phi_0^\varepsilon|+|\nabla\Phi_0^\varepsilon|).
\end{equation}
Hence with $\Omega^\delta$ denoting the boundary strip of width $\delta$, we derive from the above inequality that
\begin{equation}\label{estimate2}
\begin{aligned}
1
&\leq  C\int_{\Omega^\delta}
\bigl| \nabla^\perp g^\varepsilon\cdot\nabla\nabla^\perp g^\varepsilon\bigr| 
(\delta^{-1}|\Phi_0^\varepsilon|+|\nabla\Phi_0^\varepsilon|) \,dx
\\
&\leq C\delta^{-1}\|\nabla g^\varepsilon\|_{L^2(\Omega^\delta)}\|\nabla^2 g^\varepsilon\|_{L^2(\Omega^\delta)}\|\Phi_0^\varepsilon\|_{L^\infty}
\\
&\quad
+C\|\nabla g^\varepsilon\|_{L^4(\Omega^\delta)}\|\nabla^2 g^\varepsilon\|_{L^2(\Omega^\delta)}\|\nabla\Phi_0^\varepsilon\|_{L^4(\Omega^\delta)}.
\end{aligned}
\end{equation}
Since by Poincar\'e's inequality and $g^\varepsilon\in H_0^2(\Omega)$ we have
\begin{equation}\label{estimate3}
\|\nabla g^\varepsilon\|_{L^2(\Omega^\delta)}\leq C\delta\|\nabla^2g^\varepsilon\|_{L^2(\Omega^\delta)},
\end{equation}
the above inequality yields
\begin{align*}
1\leq C\|\nabla^2 g^\varepsilon\|_{L^2(\Omega^\delta)}^2\|\Phi_0\|_{H^2(\Omega^\delta)},
\end{align*}
where the right-hand side tends to $0$ as $\delta\rightarrow 0$. This is a contradiction. 
Thus there is a constant $C $ independent on $\lambda$ such that
$$
\|\varphi^{\delta,\varepsilon}_{\lambda}\|_{H^2(\Omega)}\leq C.
$$ 

By Leray-Schauder's fixed point theorem, the map $T^{\delta,\varepsilon}(1, \cdot)$  has a fixed point $\varphi^{\delta,\varepsilon}$  satisfying \eqref{T}.
\smallbreak

\noindent\textbf{Step 3 Passing to the limit $\varepsilon\rightarrow0$.}\\
Let $(\varphi^{\delta,\varepsilon})\in H_0^2(\Omega)$ be the solution of \eqref{T} given in Step 2.
We can follow exactly the contradiction argument   in Step 2 to show the uniform bound
$$
\|\varphi^{\delta,\varepsilon}\|_{H^2(\Omega)}\leq C,
$$
where $C$ is independent of $\varepsilon$ (by taking the subsequence $\varphi^{\delta, \varepsilon_n}$ whose norms tend to infinity by contradiction).

Hence there exists a subsequence (still denoted by $\varphi^{\delta,\varepsilon}$) such that
$$
\varphi^{\delta,\varepsilon}\rightarrow \varphi^\delta\hbox{ in }W^{1,4}(\Omega).
$$
Thus up to a subsequence $\Phi_0^{\delta, \varepsilon}+\varphi^{\delta, \varepsilon}\rightarrow \Phi_0^{\delta}+\varphi^{\delta}$ in $L^\infty(\Omega)$ and
\begin{align*}
&\rho^{\delta, \varepsilon}=\eta^\varepsilon(\Phi_0^{\delta, \varepsilon}+\varphi^{\delta, \varepsilon})\mathop{\rightharpoonup}\limits^{\ast} \rho^\delta=\eta(\Phi_0^{\delta}+ \varphi^\delta),\\
&
\mu^{\delta, \varepsilon}=b^\varepsilon(\rho^{\delta,\varepsilon})\mathop{\rightharpoonup}\limits^{\ast} \mu^\delta=b(\rho^\delta)
\quad\hbox{ in }L^\infty(\Omega),
\quad\hbox{ as }\varepsilon\rightarrow0.   
\end{align*}

Similar to \eqref{differencevarphi},
we take the difference between $\eqref{T}^\varepsilon$ and $\eqref{T}^{\varepsilon'}$   to derive  the inequality for  $ \dot\varphi^{\delta,\varepsilon}=\varphi^{\delta,\varepsilon}-\varphi^{\delta,\varepsilon'}$
 {\small\begin{align}\label{difference}
&\|\Delta \dot\varphi^{\delta,\varepsilon}\|_{L^2}^2
\\
&\leq C 
\Bigl(  \Bigl|\int_\Omega( \mu^{\delta,\varepsilon}- \mu^{\delta,\varepsilon'})
\Bigl( (\d_{22}\varphi^{\delta,\varepsilon}-\d_{11}\varphi^{\delta,\varepsilon}) (\d_{22}\dot\varphi^{\delta,\varepsilon}-\d_{11}\dot\varphi^{\delta,\varepsilon})
+ (2\d_{12}\varphi^{\delta,\varepsilon})(2\d_{12}\dot\varphi^{\delta,\varepsilon})\Bigr)\,dx\Bigr|
\notag\\
&\quad
+\Bigl| \int_{\Omega}( \rho^{\delta,\varepsilon}- \rho^{\delta,\varepsilon'})(\nabla^\perp(\Phi_0^{\delta,\varepsilon}+\varphi^{\delta,\varepsilon})
\otimes\nabla^\perp(\Phi_0^{\delta,\varepsilon}+\varphi^{\delta,\varepsilon})):\nabla\nabla^\perp\dot\varphi^{\delta,\varepsilon} \,dx\Bigr|
\notag\\
&\quad+\Bigl(\|\nabla(\Phi_0^{\delta,\varepsilon}+\varphi^{\delta,\varepsilon}-\Phi_0^{\delta,\varepsilon'}-\varphi^{\delta,\varepsilon'})\|_{L^4}\|\nabla(\Phi_0^{\delta,\varepsilon}+\varphi^{\delta,\varepsilon})\|_{L^4}
+\|f^\varepsilon-f^{\varepsilon'}\|_{H^{-1}}\Bigr)\|\dot\varphi^{\delta,\varepsilon}\|_{H^2}
\notag\\
&\quad+\Bigl| \int_{\Omega}( \mu^{\delta,\varepsilon}- \mu^{\delta,\varepsilon'})\Bigl( (\d_{22}\Phi_0^{\delta,\varepsilon}-\d_{11}\Phi_0^{\delta,\varepsilon}) (\d_{22}\dot\varphi^{\delta,\varepsilon}-\d_{11}\dot\varphi^{\delta,\varepsilon})
+ (2\d_{12}\Phi_0^{\delta,\varepsilon})(2\d_{12}\dot\varphi^{\delta,\varepsilon})\Bigr)\,dx\Bigr| \Bigr).\notag
\end{align}}
Therefore by view of the above convergence results
$$
\varphi^{\delta,\varepsilon}\rightarrow \varphi^\delta\hbox{ in }H^2(\Omega).
$$
Finally we take $\varepsilon\rightarrow0$ in \eqref{T}, then the limit $\varphi^\delta$ satisfies \eqref{ellipticvarphi}.
Hence $\Phi=\varphi^\delta+\Phi_0^\delta$ is a weak solution of~\eqref{ellipticbv}.


\subsection{More regularity results}\label{subs:k} 
In this subsection we prove the regularity results in Theorem \ref{thm} in the cases when $\eta$ is continuous and when $\eta\in C^k_b$, $k\geq2$, respectively.\\

\noindent\textbf{Case when $\eta$ is continuous}\\
If  $\Omega$ is a connected $C^{2,1}$-domain, $\Phi_0\in H^{\frac52}(\d\Omega)$ and $\Phi_1\in H^{\frac32}(\d\Omega)$, then we can extend the function $\Phi_0$ to the whole plane (still denoted by $\Phi_0$) such that $\Phi_0\in H^3(\R^2)$ with compact support and $\frac{\d\Phi_0}{\d n}|_{\d\Omega}=\Phi_1$.
Since the weak solution obtained in Subsection \ref{subs:bdd} satisfies $\Phi\in H^2(\Omega)\subset C^{\alpha}(\Omega)$, $\forall \alpha\in (0,1)$, then
$$
\rho=\eta(\Phi)\hbox{ and } \mu=b(\rho)\in C_b(\Omega),
$$
if $\eta$ is continuous.
Since $f\in L^2(\Omega;\R^2)$ and  $H^1(\Omega)\hookrightarrow L^p(\Omega)$, $\forall p\in [2,\infty)$, we can rewrite the elliptic equation \eqref{ellipticbv}   as the  fourth-order elliptic equation for $\varphi=\Phi-\Phi_0\in H^2_0(\Omega)$:
\begin{align*}
L_\mu\varphi= \nabla^\perp\cdot\div(\rho\nabla^\perp\Phi\otimes\nabla^\perp\Phi)- \nabla^\perp\cdot f
- L_\mu(\Phi_0).
\end{align*} 
By the $L^p$ estimate for the above fourth-order elliptic equation 
 in Theorem 8.6 in  \cite{Dong}, 
we have $\varphi\in W^{2,p}_0(\Omega)$ and hence $\Phi=\Phi_0+\varphi\in W^{2,p}(\Omega)$ for all finite $p$.\\

\noindent\textbf{Case when $\eta\in C^k_b$, $k\geq2$}\\ 
If  $\Omega$ is a connected $C^{k+1,1}$ domain and we assume the boundary condition $\Phi_0\in H^{k+\frac32}(\d\Omega)$, $\Phi_1\in H^{k+\frac12}(\d\Omega)$, then the above extended  function $\Phi_0 \in H^{k+2}(\R^2)\subset W^{k+1,p}(\R^2)$, $\forall p\geq 2$ with compact support and $\frac{\d\Phi_0}{\d n}|_{\d\Omega}=\Phi_1$.
We assume also smoothness in the data  $\eta, b\in C^k_b$ and $f\in H^{k-1}(\Omega)$ for $k\geq 2$.

As $\Phi\in W^{2,p}(\Omega)$ is proved in the case when $\eta$ is continuous, we first prove that $\Phi\in  W^{3,p}(\Omega)$ under the assumptions  $\Phi_0\in H^4(\R^2)\hookrightarrow W^{3,p}(\Omega)$, $f\in H^1(\R^2)\hookrightarrow L^p(\Omega)$ and $\eta,b\in C^2_b$.
We rewrite the elliptic equation \eqref{ellipticbv} as follows:
\begin{equation}\label{biharmonic}\begin{split}
 \Delta^2\varphi
&=\mu^{-1}\Bigl( -\bigl( (\d_{22}-\d_{11})\mu\bigr) \bigl((\d_{22}-\d_{11})\varphi\bigr)
-\bigl( (2\d_{12})\mu\bigr) \bigl( (2\d_{12})\varphi\bigr)
\\
&\quad 
-2(\d_2\mu)\bigl( (\d_{222}-\d_{112})\varphi\bigr)+2(\d_1\mu)\bigl( (\d_{122}-\d_{111})\varphi \bigr)
\\
&\quad 
-2(\d_1\mu) (2\d_{122})\varphi-2(\d_2\mu)(2\d_{112})\varphi-L_\mu(\Phi_0)
\\
&\quad-\nabla^\perp\cdot f+\nabla^\perp\cdot\div(\rho\nabla^\perp\Phi\otimes\nabla^\perp\Phi)\Bigr), 
\end{split}\end{equation}
where $\varphi=\Phi-\Phi_0\in H^2_0(\Omega)\cap W^{2,p}(\Omega)$.
Notice that $\rho=\eta(\Phi)$ and $\mu=(b\circ\eta)(\Phi)$ belong to $  W^{2,p}(\Omega)$ for  any $p\in (2,\infty)$.
Then for any fixed $\psi\in W^{1,q}_0(\Omega)$, $1<q<2$ we have 
\begin{align*}
&\nabla^2\mu\,\psi\in L^q(\Omega),  
\quad \nabla \mu\,\psi\in W^{1,q}_0(\Omega),
\\
& \rho\nabla^2\psi\in W^{-1,q}(\Omega),
\quad\mu^{-1}\psi\in W^{1,q}_0(\Omega),
\end{align*}
and hence the righthand side of \eqref{biharmonic} is in $W^{-1,q'}(\Omega)$, the dual space of $W^{1,q}_0(\Omega)$.
Therefore  by \eqref{biharmonic}, $\varphi\in W^{3,p}(\Omega)$ for all $p\in (2,\infty)$ and the same holds for $\Phi=\Phi_0+\varphi$. 

We assume inductively $\eta, b\in C^k_b$ and $\Phi\in W^{k,p}(\Omega)$, for $k\geq 3$, $\forall p\in (2,\infty)$, then 
$\rho=\eta(\Phi)$, $\mu=b(\rho)$ and $\varphi$ belong to $W^{k,p}(\Omega)$ for any $p\in (2,\infty)$.
Thus the righthand side of \eqref{biharmonic} belongs to $W^{k-3,p}(\Omega)$,
and hence $\varphi\in W^{k+1,p}(\Omega)$, which implies $\Phi=\Phi_0+\varphi\in W^{k+1,p}(\Omega)$.

  \appendix
  \section {The exterior domain and the whole plane cases} \label{app:ex}
In this section we consider    the Navier--Stokes system   \eqref{BVP}
\begin{equation}\label{BVP:app}
\left\{
\begin{aligned}
&\di(\rho u\otimes u)-\di(\mu Su)+\nabla\Pi= f,\\
&\di u=0,\,\di(\rho u)=0, 
\end{aligned}
\right.
\end{equation}
where $\mu=b(\rho)$, on an exterior domain or on the whole plane respectively. 
As in the bounded domain case, we will search for solutions of Frolov's form $(\rho, u)=(\eta(\Phi), \nabla^\perp \Phi)$, and it reduces to the study of the elliptic equation \eqref{BVP:Phi}  
\begin{equation}\label{BVP:Phi,app} 
\begin{aligned}
&L_\mu\Phi=-\nabla^\perp\cdot f+\nabla^\perp\cdot\div(\rho \nabla^\perp\Phi\otimes\nabla^\perp\Phi),
\end{aligned} 
\end{equation}
where
$$
L_\mu=(\d_{x_2x_2}-\d_{x_1x_1}) \mu (\d_{x_2x_2}-\d_{x_1x_1}) +(2\d_{x_1x_2})\mu (2\d_{x_1x_2} ).$$ 

Let $\Omega$ be an exterior domain or the whole plane.
We first  define the  functional spaces we are going to use
$$
D^k(\Omega):=    \dot{H}^k (\Omega)\cap \bigl(\cap_{n\geq 1}H^k(\Omega\cap B_n(0))\bigr),\quad k=1,2,
$$
where $B_n(0)$ denotes the disk centered at $0$ with radius $n$ and 
  the homogeneous Sobolev space $\dot{H}^k(\Omega)$, $k\in\N$  is  defined as
$$\dot{H}^k(\Omega)=\{g\in L^1_\loc(\Omega):\d^\alpha g\in L^2(\Omega),\,|\alpha|=k\}.$$
We define the corresponding weak solutions as follows.
\begin{definition}\label{def22}
\begin{enumerate}[(i)]
\item \textup{(Weak solutions of the Navier--Stokes equations on an exterior domain).} Let $\Omega\subset\R^2$ be the exterior domain of a bounded simply connected $C^{1}$ set. We say that a pair 
$(\rho, u)\in L^\infty(\Omega;[0,\infty))\times D^1(\Omega; \R^2)$
is a weak solution of the boundary value problem \eqref{BVP:app} with the boundary value $u_0\in H^{\frac12}({\partial\Omega};\R^2)$ and $f\in H^{-1}(\Omega;\R^2)$, if $\div u=0$, $\div(\rho u)=0$ hold in $\Omega$ in the distribution sense,  $u_0=u|_{\partial\Omega}$ is the trace of $u$ on $\partial\Omega$ and the integral identity   
\begin{equation}\label{NSweak,app}
\frac12\int_{\Omega}\mu Su:Sv\,dx=\int_{\Omega} \rho( u\otimes u):\nabla v\,dx
+\int_{\Omega} f\cdot v\,dx,
\end{equation} 
holds for all $v\in C_c^\infty(\Omega; \R^2)$ with $\div v=0$ and compact support.

\item \textup{(Weak solutions of the elliptic equation on an exterior domain).} Let $\Omega\subset\R^2$ be the exterior domain of a bounded simply connected $C^{1,1}$ set.
Let $\eta\in L^\infty(\R;[0,\infty))$ and $b\in C(\R;[\mu_\ast, \mu^\ast])$ be two given functions.
We say that   $\Phi\in D^2(\Omega; \R)$ is a weak solution of the boundary value problem \eqref{BVP:Phi,app} with the boundary values $\Phi_0\in H^{\frac32}(\partial\Omega)$, $\Phi_1\in H^{\frac12}(\partial\Omega)$,  and $f\in H^{-1}(\Omega;\R^2)$, if $\Phi_0=\Phi|_{\partial\Omega}$ and $\Phi_1=\frac{\d\Phi}{\d n}\Big|_{\d\Omega}$ in the trace sense  and the identity \eqref{ellipticweak}:
\begin{equation}\label{ellipticweak,app}\begin{split}
&\int_{\Omega}\mu
\Bigl( (\d_{x_2x_2}\Phi-\d_{x_1x_1}\Phi) (\d_{x_2x_2}\psi-\d_{x_1x_1}\psi)
+ (2\d_{x_1x_2}\Phi)(2\d_{x_1x_2}\psi)\Bigr)\,dx
\\
&=\int_{\Omega} f\cdot \nabla^\perp  \psi\,dx+\int_{\Omega}\rho(\nabla^\perp\Phi\otimes\nabla^\perp\Phi):\nabla\nabla^\perp\psi\,dx,
\end{split}
\end{equation} 
holds true for all $\psi\in C_c^\infty(\Omega; \R)$ with compact support. 

\item \textup{(Weak solutions   on $\R^2$).}  We define the weak solutions of the equations \eqref{BVP:app} (resp. \eqref{BVP:Phi,app}) on $\R^2$ as in $(i)$  (resp. in $(ii)$) above without any boundary condition.
 
\end{enumerate}
\end{definition} 

We have the following existence results on the exterior domains or on the whole plane, which will be proved in Subsection \ref{subs:ex} and Subsection \ref{subs:R2} below respectively.
\begin{theorem}
\label{thm:un}
Let $\eta\in L^\infty(\R;[0,\infty))$, $b\in C(\R;[\mu_\ast, \mu^\ast]))$, $\mu_\ast, \mu^\ast>0$ be given.  
\begin{enumerate}[(i)]
\item  \textup{(The exterior domain case).}    
Let $\Omega\subset\R^2$ be the exterior domain of a bounded simply connected $C^{1,1}$ set.
Let $f=\di F$ with  $F\in L^2(\Omega; \R^2\times\R^2)$ be given. 
Then for any $\Phi_0\in H^{\frac32}(\partial\Omega)$ and $\Phi_1\in H^{\frac12}(\partial\Omega)$,  there exists at least one weak solution $\Phi\in D^2(\Omega)$ of the boundary value problem \eqref{BVP:app}.

Let $C_0\in\R$ and $u_0\in H^{\frac12}(\partial\Omega; \R^2)$ satisfy the zero flux condition $\int_{\d\Omega}u_0\cdot n\,ds=0$. 
If $\Phi_0\in H^{\frac32}(\partial\Omega)$ and $\Phi_1\in H^{\frac12}(\partial\Omega)$ 
and $\Phi\in D^2(\Omega)$ is a weak solution of \eqref{BVP:Phi,app} given above, then the pair of Frolov's form \eqref{Fro}
\begin{equation*}
(\rho, u )=\bigl(\eta(\Phi), \,\nabla^\perp \Phi \bigr)
\end{equation*}
  is a weak solution of the boundary value problem \eqref{BVP:app} with $u\in D^1(\Omega; \R^2)$. 

\item%
\label{i:existenceL2} \textup{(The whole plane case).} 
Let $\Omega=\R^2$ and $D\subset \Omega$ be a bounded subset of positive Lebesgue measure. 
Let $f=\di F$, where $F \in L^2(\R^2; \R^2\times\R^2)$. Then for any fixed vector $d\in \R^2$, there exists at least one weak solution $\Phi \in D^2(\R^2)$ of the elliptic equation \eqref{BVP:Phi,app} on $\R^2$, such that   $u=\nabla^\perp\Phi\in D^1(\R^2; \R^2)$ is a weak solution of the equation \eqref{BVP:app} on $\R^2$ and $\frac{1}{\hbox{meas}(D)}\int_{D}u=d$.

\end{enumerate} 
\end{theorem}  
\begin{remark}[Asymptotic behaviours  on unbounded domains]
If $\Omega$ is an unbounded  domain, we denote the  asymptotic behavior  of the solutions $u$ at infinity by $u_{\infty}$
$$\lim_{|x|\to \infty}u(x)=u_{\infty}, \quad u_{\infty}\in\R^2.$$
The  existence result in Theorem \ref{thm:un} does not give the information of $u_\infty$. 
On the other side, we   don't know the existence of decaying solutions of the Navier-Stokes system \eqref{BVP:app} on the exterior domain or the whole plane. 

These problems are still not clear even for the classical Navier--Stokes equations \eqref{CNS}
\begin{equation}\label{CNS:app}
\left\{
\begin{aligned}
&\di(  u\otimes u)-\nu \Delta u+\nabla\Pi= f,\\
&\di u=0,
\end{aligned}
\right.
\end{equation}
  for the two-dimensional case. 
In the three-dimensional case, thanks to the Hardy inequality
\begin{equation}
\label{3D:Hardy}
\left\|\frac{u-u_\infty}{|x|}\right\|_{L^2(\Omega)}\le C \|\nabla u\|_{L^2(\Omega)}, \quad \Omega\subset\R^3,   
\end{equation}
one has 
$u-u_\infty\in L^6(\Omega)$ and one can show the existence of the weak solutions with certain limit at large distance.
However, the Hardy inequality \eqref{3D:Hardy} does not hold for the two dimensional case, which brings more difficulties to study the asymptotic behaviours (the limit of $u$ at large distance is hard to be determined if $u$ is only in $\dot H^1(\R^2)$).
There are some partial results in the two-dimensional case: 
The solvability  of \eqref{CNS:app}
on the exterior domain with the restriction $u_{\infty}=0$ (under some symmetric assumptions) was established in e.g. \cite{Wittwer, Yamazaki}.
There are also some works considering the asymptotic behaviors of the (general) weak solutions: \textcite {Gilbarg,Gilbarg2} showed that the weak solutions of \eqref{CNS:app} satisfy  $\lim_{r\to\infty}\int_0^{2\pi}|u(r,\theta)|^2 \,d\theta=\infty$ or $\lim_{r\to\infty}\int_0^{2\pi}|u(r,\theta)-\bar u|^2\,d\theta=0$ for some $\bar u\in \R^2$, and J. Amick discussed the relation between $u_\infty$ and $\bar u$ in \cite{Amick}. 

\textcite{Galdi} showed the  non uniqueness of the solutions to the classic Navier--Stokes equation \eqref{CNS:app} with certain boundary condition $u_0$ and $u_{\infty}=0$. 
Hence the weak solutions of the system \eqref{BVP:app} are also not unique, at least in the case without any smallness or symmetric assumptions.


\end{remark}


 \subsection{The exterior domain case}\label{subs:ex}
 
In this subsection we prove (i) in Theorem \ref{thm:un}.
We are going to show the existence of the weak solutions to the boundary value problem of the elliptic system \eqref{BVP:Phi,app} on the exterior
domain of a simply connected $C^{1,1}$-set $\Omega$, by an approximation
argument.
Then the existence of the weak solutions of Frolov's form of the boundary value  problem for the Navier--Stokes equations \eqref{BVP:app} follows immediately.

Let $N\in \N$ such that $\Omega^c\subset B_N(0)=\{x\in\R^2\,|\, |x|< N\}$. 
Let $\Omega_n=\Omega\cap B_{N+n}(0) \subset \R^2$, then $\{\Omega_n\}$ is a monotonically increasing sequence which has the exterior domain $\Omega$ as its limit. 
By the solvability result in Theorem \ref{thm} $(i)$, for any given $\eta\in L^\infty(\R;[0,\infty))$, $b\in C(\R;[\mu_\ast, \mu^\ast])$, $0<\mu_\ast\leq \mu^\ast$, and $f=\div F\in H^{-1}(\Omega; \R^2)$, there exists a weak solution  $\Phi_n\in H^2(\Omega_n)$ of the boundary value problem \eqref{BVP:Phi,app} on $\Omega_n$ with the boundary condition $\Phi_n|_{\d\Omega}=\Phi_0\in H^{\frac32} (\d\Omega)$, $\frac{\d\Phi_n}{\d n}|_{\d\Omega}=\Phi_1\in H^{\frac12} (\d\Omega)$, and $\Phi_n|_{\d B_{N+n}(0)}=\frac{\d\Phi_n}{\d n}|_{\d B_{N+n}(0)}=0$.  
According to the proof of Theorem \ref{thm} $(i)$ in   Subsection \ref{subs:bdd}, for any fixed small enough $\delta>0$, we can write 
$$\Phi_n=\varphi_n^\delta+\Phi_0^\delta, \hbox{ with }
\varphi_n^\delta\in H^2_0(\Omega_n) \hbox{ satisfying \eqref{ellipticvarphi},}$$
 and   $\Phi_0^\delta(x)=\Phi_0(x)\zeta(x;\delta)$  is defined in \eqref{Phi0delta}.
We extend $\varphi_n^\delta$ from $\Omega_n$ to $\Omega$ by simply taking $\varphi_n^\delta|_{\Omega\backslash\Omega_n}=0$ (still denoted by $\varphi_n^\delta$). 

We are going to show that $\|\varphi_n^\delta\|_{\dot{H}^2(\Omega)}$ is uniformly bounded.
We take $\psi=\varphi_n^\delta$ in the equation \eqref{ellipticvarphi} for $\varphi_n^\delta$, to derive 
\begin{equation*}\label{T2}\begin{split}
&\int_{\Omega}\mu_n
\Bigl( (\d_{22}\varphi_n^\delta-\d_{11}\varphi_n^\delta)^2 
+ (2\d_{12}\varphi_n^\delta)^2\big)\,dx
\\
&=\int_{\Omega}\rho_n(\nabla^\perp(\Phi_0^{\delta}+\varphi_n^\delta)
\otimes\nabla^\perp(\Phi_0^{\delta}+\varphi_n^\delta)):\nabla\nabla^\perp\varphi_n^\delta\,dx
 -\int_{\Omega}F\cdot \nabla \nabla^\perp\varphi_n^\delta\,dx 
\\
& \quad
-\int_{\Omega}\mu_n\Bigl( (\d_{22}\Phi_0^{\delta}-\d_{11}\Phi_0^{\delta}) (\d_{22}\varphi_n^\delta-\d_{11}\varphi_n^\delta)
+ (2\d_{12}\Phi_0^{\delta})(2\d_{12}\varphi_n^\delta)\Bigr)\,dx,
\end{split}
\end{equation*} 
where $\rho_n=\eta(\Phi_n)=\eta(\varphi_n^\delta+\Phi_0^\delta)$ and $\mu_n=b(\rho_n)$.
 Similarly as in the derivation of \eqref{Estimate:varphi,eps}, we have  
\begin{equation}\label{Delta}
\begin{aligned}
\|\Delta \varphi_n^\delta\|_{L^2(\Omega)}^2
&\leq C(\rho^\ast,\mu_\ast,\mu^\ast)
\Bigl(\bigl( \|\Phi_0^{\delta}\|_{H^2}^2+ \|F\|_{L^2}+\|\Phi_0^{\delta}\|_{H^2}\bigr)\|\Delta \varphi_n^\delta\|_{L^2(\Omega)}
\\
&\qquad+\int_{\Omega} \rho_n\nabla^\perp \varphi_n^\delta\cdot\nabla\nabla^\perp\varphi_n^\delta\cdot
 \nabla^\perp\Phi_0^{\delta} \,dx\Bigr).
\end{aligned}
\end{equation}
 By the Riesz inequality (cf.  \cite{Demengel}), we have $\|\Delta\varphi_n^\delta\|_{L^2(\Omega)}\sim \|\varphi_n^\delta\|_{\dot H^2(\Omega)}$.
We are going to follow exactly the contradiction argument  in Step 3 in Subsection \ref{subs:bdd} to show the uniform boundedness of $\|\varphi_n^\delta\|_{\dot H^2(\Omega)}$ and hence we will just sketch the proof and emphasize the difference for the exterior domain case. 
Suppose by contradiction that there exists a subsequence $(\varphi_{k_n}^\delta)\subset(\varphi_n^\delta)$ such that
$$
\|\Delta\varphi_{k_n}^\delta\|_{L^2(\Omega)}\rightarrow\infty,\quad\text{as}\quad k_n\to\infty.
$$
Denote $g_{k_n}^\delta=\frac{\varphi_{k_n}^\delta}{\|\Delta\varphi_{k_n}^\delta\|_{L^2(\Omega)}}$, then  $\|\Delta g_{k_n}^\delta\|_{L^2(\Omega)}=1$, $\operatorname{tr}(g_{k_n}^\delta)|_{\d\Omega}=0$ and there exist  a subsequence (still denoted by $(g_{k_n}^\delta)$) and $g\in\dot{H}^2(\Omega)$ with $\operatorname{tr}(g)|_{\d\Omega}=0$ such that
\begin{align*}
g_{k_n}^\delta\rightharpoonup g \hbox{ in }\,\dot{H}^2(\Omega),\quad \text{as}\quad k_n\to\infty. 
\end{align*} 
Here the limit function $g$ does not depend on $\delta$.
Recall that $\Omega^\delta$ is the boundary strip of width $\delta$. 
By Poincar\'e's inequality we obtain $g_{k_n}^\delta|_{\Omega^\delta}\rightharpoonup g|_{\Omega^\delta} $ in $H^2(\Omega^\delta)$ and by Sobolev embedding $g_{k_n}^\delta|_{\Omega^\delta}\to g|_{\Omega^\delta} $ in $W^{1,4}(\Omega^\delta)$.
We take $k_n\rightarrow\infty$ in \eqref{Delta} to derive that
\begin{align*}
1\leq  C\int_{\Omega^\delta}
\Bigl| \nabla^\perp g\cdot\nabla\nabla^\perp g\cdot
 \nabla^\perp \Phi_0^\delta\Bigr|\,dx.
\end{align*}
By using the same estimates \eqref{estimate1}-\eqref{estimate2}-\eqref{estimate3}   we arrive at
\begin{align*}
1\leq C\|\Delta g\|_{L^2(\Omega^\delta)}^2\|\Phi_0\|_{H^2},
\end{align*}
where the right-hand side tends to $0$ as $\delta\rightarrow 0$. This is a contradiction. Hence
there exists a constant $C$ independent of $n$ such that
$$
\|\varphi_n^\delta\|_{\dot{H}^2(\Omega)}\leq C.
$$
Then  there exists a subsequence (still denote by $(\varphi_n^\delta)$) converging weakly to a limit $\varphi^\delta$ in $\dot H^2(\Omega)$, with $\operatorname{tr}|_{\partial\Omega}(\varphi^\delta)=0$. 
Let 
$$\Phi=\Phi_0^\delta+\varphi^\delta,$$
 then $\Phi_n=\Phi_0^\delta+\varphi_n^\delta\rightharpoonup\Phi$ in $ \dot{H}^2(\Omega)$.
 By Poincar\'e's inequality and a Cantor diagonal argument, there exists a subsequence (still denoted by $(\Phi_n)$) such that
 $$
 \Phi_n\rightarrow \Phi \hbox{ a.e. in }\Omega \hbox{ and } \rho_n\mathop{\rightharpoonup}\limits^\ast \rho=\eta(\Phi),\, \mu_n\mathop{\rightharpoonup}\limits^\ast b(\rho)=\mu
\hbox{ in }L^\infty(\Omega)\hbox{ as }n\rightarrow\infty.
 $$

We are going to show that $\Phi$ is a  weak solution of the equation \eqref{BVP:Phi,app}  on the exterior domain $\Omega$. 
Fix a test function $\psi\in C_c^\infty(\Omega)$ with compact support. 
Then there exists a ball containing   $ \Omega^C\cup \operatorname{Supp}(\psi)$ and without loss of generality we suppose it to be $B_1(0)$.
Let $V=B_1(0)\cap\Omega$, then, we are going to show, up to a subsequence,
$$
\varphi_n^\delta\rightarrow \varphi^\delta\hbox{ in }H^2(V).
$$
Indeed, we take a smooth cutoff function $\chi$ with $\chi=1$ on $B_1(0)$ and $\chi=0$ outside $B_2(0)$.
We take the difference between the equation \eqref{ellipticvarphi} for $\varphi_n^\delta$ and the equation \eqref{ellipticvarphi} for $\varphi_m^\delta$ and then take $\psi=\chi\varphi_{n,m}^\delta$, $\varphi_{n,m}^\delta=\varphi_n^\delta-\varphi_m^\delta$ in \eqref{ellipticvarphi}.
We arrive at the following inequality similar as \eqref{difference}
{\small\begin{align*} 
&\int_{B_2(0)\cap \Omega}\mu_n \Bigl( (\d_{22}-\d_{11})\varphi_{n,m}^\delta(\d_{22}-\d_{11})(\chi\varphi_{n,m}^\delta)+2\d_{12}(\varphi_{n,m}^\delta)2\d_{12}(\chi\varphi_{n,m}^\delta)\Bigr)
\\
&\leq    \Bigl|\int_{B_2(0)\cap \Omega}(\mu_n-\mu_m)
\Bigl( (\d_{22}-\d_{11})\varphi_m^\delta(\d_{22}-\d_{11})(\chi\varphi_{n,m}^\delta)+2\d_{12}(\varphi_m^\delta)2\d_{12}(\chi\varphi_{n,m}^\delta)\Bigr)\,dx\Bigr|
\notag\\
&\quad
+\Bigl| \int_{B_2(0)\cap \Omega}\Bigl( \rho_n\nabla^\perp\Phi_n\otimes\nabla^\perp\Phi_n-\rho_m\nabla^\perp\Phi_m\otimes\nabla^\perp\Phi_m\Bigr):\nabla\nabla^\perp(\chi\varphi_{n,m}^\delta)\,dx
\\
&\qquad
-\int_{B_2(0)\cap \Omega}(\mu_n-\mu_m)
\Bigl( (\d_{22} -\d_{11})\Phi_0^\delta (\d_{22}-\d_{11})(\chi\varphi_{n,m}^\delta)
+ (2\d_{12}\Phi_0^\delta)2\d_{12}(\chi\varphi_{n,m}^\delta)\Bigr)\,dx\Bigr|.
\end{align*}} 
The left-hand side above is bigger than
\begin{align*}
&\int_{V}\mu_n \Bigl( \bigl( (\d_{22}-\d_{11})\varphi_{n,m}^\delta\bigr)^2+\bigl(2\d_{12}(\varphi_{n,m}^\delta)\bigr)^2\Bigr)
\\
&-\Bigl| \int_{B_2(0)\backslash B_1(0)} \mu_n\Bigl( (\d_{22}-\d_{11})\varphi_{n,m}^\delta 
\bigl(( (\d_{22}-\d_{11})\chi) \varphi_{n,m}^\delta+2(\d_2\chi\d_2\varphi_{n,m}^\delta-\d_1\chi \d_1\varphi_{n,m}^\delta)\bigr)
\\
&\qquad
+2\d_{12}\varphi_{n,m}^\delta(2(\d_{12}\chi)\varphi_{n,m}^\delta+2\d_1\chi\d_2\varphi_{n,m}^\delta+2\d_2\chi\d_1\varphi_{n,m}^\delta\bigr)\Bigr)\Bigr|.
\end{align*}
As up to a subsequence we may assume  
\begin{align*}
&\varphi_{n,m}^\delta\rightarrow 0\hbox{ in }H^1(B_2(0)\cap\Omega),
\quad \Phi_{n}-\Phi_m\rightarrow 0\hbox{ in }W^{1,4}(B_2(0)\cap\Omega)
\hbox{ as }n,m\rightarrow\infty,
\end{align*}
we hence have $\varphi_{n,m}^\delta\rightarrow 0$ in $H^2(V)$.
Therefore $\Phi_n\rightarrow\Phi$ in $H^2(V)$, and the limit $\Phi$ (together with the limits $\rho,\mu$) satisfies the integral equality \eqref{ellipticweak,app}.
As $\psi\in C^\infty_c(\Omega)$ has been chosen arbitrarily, $\Phi\in D^2(\Omega)$ is a weak solution of equation \eqref{BVP:Phi,app} on $\Omega$ satisfying the boundary condition $\Phi|_{\d\Omega}=\Phi_0\in H^{\frac32} (\d\Omega)$, $\frac{\d\Phi}{\d n}|_{\d\Omega}=\Phi_1\in H^{\frac12} (\d\Omega)$ in the trace sense.

\subsection{The whole plane case}\label{subs:R2}
In this subsection, we will follow the idea in \cite{Guillod}  to prove $(ii)$ in Theorem \ref{thm:un}. 
We will denote $\fint_D=\frac{1}{\hbox{meas}(D)}\int_{D}$.
We  take a bounded simply connected $C^{1,1}$-domain $U\supset D$ and we make an Ansatz
$$
u=d+w-\bar w,
$$
where $w\in H_0^1(U; \R^2)$, $\div w=0$ and $\bar w= \fint_{D} w\in \R^2$ is a constant vector. 
In other words, if $\varphi$ is the stream function of $w$, then 
$$
u=\nabla^\perp \Phi, 
$$
where we take 
$$
\Phi=\varphi+(d-\bar w)\cdot \begin{pmatrix}x_2\\-x_1\end{pmatrix}+C,\quad \bar w= \fint_{D} \nabla^\perp\varphi\in\R^2
$$
with any fixed constant $C\in \R$.
We can typically choose 
$$
C=C_\varphi:=-\fint_D\varphi \,dx-\fint_D (d-\overline w)\cdot\begin{pmatrix}x_2\\ -x_1\end{pmatrix}dx,
$$
such that $\fint_D \Phi=0$.

We then search for $\varphi\in H_0^2(U)$ satisfying 
\begin{equation*}
\begin{split}
&\int_{U}\mu
\Bigl( (\d_{22}\varphi-\d_{11}\varphi) (\d_{22}\psi-\d_{11}\psi)
+ (2\d_{12}\varphi)(2\d_{12}\psi)\Bigr)\,dx
\\
&=\int_{U}\rho(\nabla^\perp \Phi
\otimes\nabla^\perp \Phi):\nabla\nabla^\perp\psi\,dx
-\int_{U}F\cdot \nabla\nabla^\perp\psi\,dx,
     \quad \forall\psi\in H  ^2_0(U;\R),
\end{split}
\end{equation*} 
where $\Phi=\varphi+(d- \fint_{D} \nabla^\perp\varphi)\cdot \begin{pmatrix}x_2\\-x_1\end{pmatrix}+C_\varphi$, $\rho=\eta( \Phi)$ and $\mu=b(\rho)$. 
Following the proof lines in Subsection \ref{subs:bdd}, such  $\varphi$ exists, and hence there exists $w\in H_0^1(U; \R^2)$  satisfying 
\begin{equation}\label{ellipticgamma}
\frac12\int_{U}\mu
Sw:Sv\,dx=\int_{U}\rho\big(w+d-\fint_D w\big)
\otimes\bigl(w+d-\fint_Dw\bigr):\nabla v\,dx
-\int_{U}F\cdot \nabla v\,dx
\end{equation}
for any $v\in H_0^1(U;\R^2) $ with $\di v=0$. 
By taking $v=w$ in \eqref{ellipticgamma}, the first integral on the righthand side vanishes (since $\div(\rho(w+d-\fint_{D}w))=0$), and we obtain 
\begin{equation*}\label{unibound}
 \|w\|_{\dot{H}^1(U)}\le C( \mu_\ast)\|F\|_{L^2(\R^2)}.   
\end{equation*} 
Then we arrive at a distribution  solution $\Phi=\varphi+(d- \fint_{D} \nabla^\perp\varphi)\cdot \begin{pmatrix}x_2\\-x_1\end{pmatrix}+C_\varphi$, with $\varphi\in H_0^2(U)$ of the elliptic equation \eqref{BVP:Phi,app} on the bounded domain $U$,  which satisfies  $\fint_D\Phi=0$.
Hence $(\rho, u)=(\eta( \Phi), \nabla^\perp \Phi)=(\eta(\Phi), w+d-\fint_D w)$ with $w=\nabla^\perp\varphi\in H_0^1(U)$ is a distribution solution  of the system \eqref{BVP:app} on the bounded domain $U$, which satisfies  $\fint_D u=d$.

As in Subsection \ref{subs:ex}, we take the approximation argument to show the existence of the weak solutions on the whole plane $\R^2$.
Indeed, if we take $U=B_n(0)\supset D$ in the above,
then we have arrived at a distribution solution 
$$ \Phi_n=\varphi_n+(d-\overline{w_n})\cdot \begin{pmatrix}
x_2\\ -x_1
\end{pmatrix}+C_{\varphi_n},$$
with $\varphi_n\in H_0^2(B_n(0))$ and $ \overline{w_n}= \fint_{D} \nabla^\perp\varphi_n\in\R^2$,
of   \eqref{BVP:Phi,app}   in $B_n(0)$ with $\fint_D \Phi_n=0$,
and hence a distribution solution
$$
(\rho_n, u_n)=(\eta( \Phi_n), \nabla^\perp  \Phi_n)=(\eta( \Phi_n),   w_n+d-\overline {w_n}),$$ 
with $w_n=\nabla^\perp \varphi_n\in H^1_0(B_n(0); \R^2)$, 
 of \eqref{BVP:app} in $B_n(0)$ with $ \fint_{D} u_n=d$. 
We extend $w_n$ trivially to $\R^2$ (that is, we simply take $w_n=0$ outside $B_n(0)$) and take  $\Phi_n= (d-\overline{w_n})\cdot \begin{pmatrix}
x_2\\ -x_1
\end{pmatrix}+C_{\varphi_n}$ and $u_n=d-\overline{w_n}$ outside $B_n(0)$, respectively.
 Let 
 $$\tau_n=w_n-\overline{w_n} \hbox{ with }\fint_{D}\tau_n=0,$$
then $u_n=d+\tau_n,$ with
\begin{equation*}
\|\tau_n\|_{\dot{H}^1(\R^2)}= \|w_n\|_{\dot{H}^1(\R^2)}\le C(\mu_*)\|F\|_{L^2(\R^2)}.
\end{equation*} 

Let $v\in C_c^{\infty}(\R^2; \R^2)$ with $\di v=0$ be any test function, then there exists $N\in \N$ such that $\operatorname{Supp}(v)\cup D\subset B_N(0)$. 
By the above uniform bound on $(\tau_n)$ and Poincar\'e's inequality, there exists a subsequence (still denoted by $(\tau_n)$)  such that $\tau_n\rightharpoonup\tau$ in $\dot H^1(\R^2)$ as $n\rightarrow\infty$, and in $H^1(B_N(0))$. 
Thus $\{u_n\}$ is uniformly bounded in $H^1(B_N(0))$ and $u_n\rightharpoonup u=\tau+d$ in $H^1(B_N(0))$.
Since $\fint_D \Phi_n=0$, by Poincar\'e inequality again, $\{\Phi_n\}$ is uniformly bounded in $H^2(B_N(0))$, and  up to a subsequence   $  \Phi_n\rightharpoonup \Phi$ in $H^2(B_N(0))$, with $\nabla^\perp \Phi=u$ and $\fint_D \Phi=0$.
Thus $  \Phi_n\rightarrow \Phi$ in $W^{1,4}(B_N(0))\subset C^{1/2}(B_N(0))$,
and $\rho_n=\eta( \Phi_n)\mathop{\rightharpoonup}\limits^\ast \rho=\eta(\Phi)$, $\mu_n=b(\rho_n)\mathop{\rightharpoonup}\limits^\ast \mu=b(\rho)$ in $L^\infty(B_N(0))$.
Exactly as the end of Subsection \ref{subs:ex}, $u_n\rightarrow u$ in $H^1(B_N(0))$.
Thus the limits $u,\,\rho,\,\mu$ satisfy the integral equality \eqref{NSweak,app} for the given test function $v$, and hence $u\in D^1(\R^2; \R^2)$ with $\fint_D u=d$ is a weak solution of the Navier-Stokes equation \eqref{BVP:app} on $\R^2$.
Correspondingly $\Phi\in D^2(\R^2)$ is a weak solution of the elliptic equation \eqref{ellipticweak,app} on $\R^2 $.

\section{Explicit symmetric solutions with piecewise-constant viscosity coefficients}\label{example} 
In this section, we will give explicit examples of solutions with piecewise-constant viscosity coefficients  of the Navier--Stokes equations
\begin{equation}\label{SNS:app}
\left\{
\begin{aligned}
&\di(\rho u\otimes u)-\di(\mu Su)+\nabla\Pi= 0,\\
&\di u=0,\,\di(\rho u)=0.
\end{aligned}
\right.
\end{equation}
Those examples are  
the parallel, concentric and radial flows
in Theorem \ref{thm2}, where the density functions have the forms
\begin{align*}
\rho=\rho(x_2)\hbox{ in }\R^2,
\,\,\rho(r)\hbox{ in }\R^2\backslash\{0\},
\,\,\hbox{ and }\,\,\rho(\theta)\hbox{ in }\R^2\backslash\{0\},
\end{align*}
and the corresponding 
velocity fields have the forms
\begin{equation*}
u=u_1(x_2)\,e_1\hbox{ in }\R^2,
\quad
rg(r)\,e_\theta\hbox{ in }\R^2\backslash\{0\},
\,\hbox{ and }\,
\frac{h(\theta)}{r}e_r\hbox{ in }\R^2\backslash\{0\}.
\end{equation*}
The functions  $u_1, g, h$ solve the following ODEs
\begin{align}
\label{ODE:app}
&
 \partial_{x_2}(\mu\partial_{x_2}u_1) =C,\notag
 \\
 &  \d_r( \mu r^3 \d_r g) = -Cr,\\
 &\rho h^2 +\d_\theta(\mu \d_\theta h)+4(\mu h)=C, \notag
\end{align}
where   $C\in\R$ can be any real number.
We assume that $\mu=b(\rho)$ depends strictly monotone on $\rho$.

\medskip
\noindent\textbf{Examples of parallel flows}

If $(\rho, u)=\bigl( \rho(x_2), \,\, u_1(x_2)\,e_1  \bigr)$ (not necessarily $\rho'\neq 0$) solves the system \eqref{SNS:app}, then $u$ satisfies $\eqref{ODE:app}_1$: $\d_{x_2}(\mu \d_{x_2}u_1)=C\in\R$. 
In particular,   with   the following viscosity coefficient $\mu$
\begin{equation*}
\label{mu:x2}
\mu=\mu(x_2)=2\mathbb 1_{\left\{x_2>0\right\}}+\mathbb 1_{\left\{x_2\le 0\right\}},
\end{equation*}
we have for some constant $C_1\in\R$ that
\begin{equation*}\label{solutionxy}  
 \partial_{x_2}u= {u_1}'(x_2)\,e_1
 =\Big( \bigl(\frac{C}{2}x_2+\frac{C_1}{2} \big)\mathbb{1}_{\{x_2>0\}}+\big(Cx_2+C_1 \big)\mathbb{1}_{\{x_2\leq  0\}} \Bigr) e_1,
\end{equation*} 
and hence
\begin{align*} 
&u_1''(x_2)=\frac C2\mathbb{1}_{\{x_2>0\}}+C\mathbb{1}_{\{x_2< 0\}}
-\frac{C_1}{2}\delta_0(x_2).
\end{align*} 

There exists a real constant $C_2\in\R$ such that $u\in H^1_\loc(\R^2)$ reads as  
\begin{equation}\label{u1:x2}
u=\Bigl( \bigl(\frac{C}{4}x_2^2+\frac{C_1}{2}x_2+C_2\bigr)\mathbb{1}_{\{x_2>0\}}+\bigl(\frac{C}{2}x_2^2+C_1x_2+C_2\bigr)
\mathbb{1}_{\{x_2\le 0\}}\Bigr) e_1.
\end{equation}
If we consider the Couette flow on the strip $\R\times [-1,1]$ with the boundary conditions
\begin{equation}\label{BC:x2}
u|_{\R\times\{\pm1\}}= a_{\pm}\, e_1\in\R^2, 
\end{equation}
then there hold only two equations for the three constants $C, C_1, C_2$  
$$
C=4(a_--a_+)+6C_1,\quad C_2=2a_+-a_--2C_1,\quad C_1\in\R.
$$
Hence  there are uncountably many solutions with the density function 
\begin{equation}\label{rho:x2}
\rho(x_2)=b^{-1}(2)\mathbb{1}_{\{x_2>0\}}+b^{-1}(1)\mathbb{1}_{\{x_2\leq 0\}}, 
\end{equation}
and the velocity vector field \eqref{u1:x2} to the boundary value problem \eqref{SNS:app}-\eqref{BC:x2}.
\footnote{For the homogeneous flow $\mu=1$,  the velocity vector field in the form of $u_1(x_2)\,e_1$ reads as $u=\bigl(\frac C2x_2^2+C_1x_2+C_2 \bigr) e_1$
with $C_1=\frac{a_+-a_-}{2}$, $\frac C2+C_2=\frac{a_++a_-}{2}$.}

It is easy to see that if $a_+<a_-<2a_+$ and $0<C_1<\frac{2a_+-a_-}{ {2}}$, then $C,C_2>0$ and $u_1(x_2)>0$ for $x_2\in [-1,1]$. 
Hence  $\d_{x_2}\Phi=u_1>0$ and there exists a constant $C_3\in\R$ such that the stream function
\begin{align*}
\Phi=\bigl( \frac{C}{12}x_2^3+\frac{C_1}{4}x_2^2+C_2 x_2+C_3 \bigr)\mathbb{1}_{\{x_2>0\}}+\bigl( \frac{C}{6}x_2^3+\frac{C_1}{2}x_2^2+C_2 x_2+C_3 \bigr)\mathbb{1}_{\{x_2\le 0\}}
\end{align*}
is a strictly increasing function from $[-1,1]$ to $[\Phi_-, \Phi_+]$, where 
\begin{align*}
\Phi_-=\Phi(-1)=\frac32 C_1+C_3-\frac43a_++\frac13a_-,
\quad
\Phi_+=\Phi(1)=-\frac54C_1+C_3+\frac53a_+-\frac23a_-.
\end{align*}  
Then the pair \eqref{rho:x2}-\eqref{u1:x2} is a solution of the system \eqref{SNS:app} in the Frolov's form $(\rho, u)=(\eta(\Phi), \nabla^\perp\Phi)$ with
\begin{align*}
\eta(y)=\left\{\begin{array}{ll}b^{-1}(2)&\hbox{ if }y\in(C_3,\Phi_+],\\ b^{-1}(1)&\hbox{ if }y\in[\Phi_-, C_3].\end{array}\right.
\end{align*}


\medskip
\noindent\textbf{Examples of concentric flows}

If $(\rho, u)= ( \rho(r), \,  rg(r)e_\theta)$ (not necessarily $\rho'\neq 0$) solves the system \eqref{SNS:app} with $f=0$, then $u$ satisfies $\eqref{ODE:app}_2$: $ \d_r( \mu r^3 \d_r g) = -Cr$. 
In particular,   with   the following viscosity coefficient $\mu$  
\begin{equation*}
\label{radial:murho}
\mu=\mu(r)=2\mathbb 1_{\{0<r<1\}}+\mathbb 1_{\{r\geq 1\}},
\end{equation*}   
we have for some real constant $C_1\in\R$ that
\begin{equation*}\label{radial:g}
\d_r g=\frac{-\frac C2r^2+C_1}{\mu r^3}=\bigl(-\frac C4\frac1r+\frac{C_1}{2}\frac1{r^3}\bigr)\mathbb 1_{\{0<r<1\}}
+\bigl(-\frac C2\frac1r+C_1\frac{1}{r^3}\bigr)\mathbb 1_{\{r\geq 1\}},
\,\, C,C_1\in\R.
\end{equation*}
There exists a constant $C_2\in\R$ such that (for $u\in H^1_\loc(\R^2\backslash\{0\})$)
\begin{equation}\label{g,r}
g(r)=\bigl(-\frac C4\ln r-\frac{C_1}{4}(\frac{1}{r^2}-1)+C_2\bigr)\mathbb 1_{\{0<r<1\}}
+\bigl( -\frac C2\ln r-\frac{C_1}{2}(\frac{1}{r^2}-1)+C_2\bigr)\mathbb 1_{\{r\geq 1\}}.
\end{equation}
If we consider the concentric flow on the annulus $\{x\in\R^2\,|\,\frac12\leq |x|\leq2\}$ and suppose the boundary conditions
\begin{equation}\label{BC:r}
u|_{\{x| |x|=\frac12\}}=\frac12 g_- e_\theta|_{\{x| |x|=\frac12\}},
\quad u|_{\{x| |x|=2\}}=2g_+ e_\theta|_{\{x| |x|=2\}},
\end{equation}
then
$$
C=\left(\frac{3\ln2}{4}\right)^{-1}(\frac98C_1-g_++g_-),
\quad C_2=\frac13(\frac98C_1+g_++2g_-),
\quad C_1\in\R.
$$
Hence the density function 
\begin{equation*}
\rho=\rho(r)=b^{-1}(2)\mathbb{1}_{\{0<r<1\}}+b^{-1}(1)\mathbb{1}_{\{r\geq1\}}
\end{equation*}
 and the velocity vector field $u=rge_\theta$ with $g$ given in \eqref{g,r} is a solution  of the boundary value problem \eqref{SNS:app}-\eqref{BC:r}.
We can follow the argument at the end of Case $\rho=\rho(x_2)$ to find the function $\eta$ such that $\rho=\eta(\Phi)$, provided with more  restrictions on $g_-, g_+, C_1$. We leave this to interested readers.

\medskip
\noindent\textbf{Examples of radial flows}

If $(\rho, u)= ( \rho(\theta), \,  \frac{h(\theta)}{r}e_r)$ (not necessarily $\rho'\neq 0$) solves the system \eqref{SNS:app}, then $u$ satisfies $\eqref{ODE:app}_3$: $\rho h^2 +\d_\theta(\mu \d_\theta h)+4(\mu h)=C\in\R$.  
Let the viscosity coefficient $\mu$   be
\begin{equation}\label{mu,theta} 
\mu=\mu(\theta) = 2\mathbb{1}_{[0, \frac\pi4)}+\mathbb{1}_{[\frac\pi4,\frac\pi2]}.
\end{equation} 
Then $(\rho, u)= ( \rho(\theta), \,  \frac{h(\theta)}{r}e_r)$ with $h(\theta)$ satisfying the following 
$$\d_\theta(\mu\d_\theta h)=0,
\quad \rho h+4\mu=0$$
is a solution of \eqref{SNS:app}, and in particular $h,\rho$ can be taken as follows
\begin{align*}
&\d_\theta h=\frac{-2}{\mu}=-\mathbb{1}_{[0, \frac\pi4)}-2\mathbb{1}_{[\frac\pi4,\frac\pi2]},
\quad h=(-\frac\pi2- \theta)\mathbb{1}_{[0,\frac\pi4)}+(-\frac\pi4- 2\theta )\mathbb{1}_{[\frac\pi4,\frac\pi2]},\notag
\\
& \rho=-\frac{4\mu}{h}\hbox{ such that }\mu=b(\rho)\hbox{ holds}.
\end{align*}
This radial flow moves toward the origin and moves faster when closer to the origin.
There are obviously other solutions of form $(\rho, u)= ( \rho(\theta), \,  \frac{h(\theta)}{r}e_r)$ to the system \eqref{SNS:app} with the viscosity coefficient \eqref{mu,theta}, and we do not go to details here.


\section*{Acknowledgments}
The work of Xian Liao is funded   by the Deutsche Forschungsgemeinschaft (DFG, German Research Foundation) – Project-ID 258734477 – SFB 1173.

\printbibliography
\end{document}